\documentclass[a4paper]{article}

\usepackage[a4paper, textwidth=130mm]{geometry}

\usepackage{url}

\makeatletter
\AtBeginDocument{%
  \def\doi#1{\url{https://doi.org/#1}}}
\makeatother

\usepackage{amsmath}
\usepackage{amsthm}
\usepackage{amsfonts}

\usepackage{caption}
\usepackage[english]{babel}

\usepackage[numbers,sort&compress]{natbib}

\usepackage{algpseudocode, algorithm}

\usepackage{pgfplots}
\usepgfplotslibrary{groupplots}
\pgfplotsset{compat=newest}
\usetikzlibrary{fit}

 \pgfplotscreateplotcyclelist{own cycle}{
	every mark/.append style={fill=white},mark=*\\
	every mark/.append style={fill=white},mark=triangle*\\
	every mark/.append style={fill=white},mark=oplus*\\
	every mark/.append style={fill=white},mark=square*\\
	every mark/.append style={fill=white},mark=otimes*\\
	every mark/.append style={fill=white},mark=diamond*\\
}

\pgfplotsset{cycle list name=own cycle}

\usepackage{mathtools}
\usepackage{hyperref}
\usepackage[capitalise]{cleveref}

\usepackage{authblk}

\newcommand{\norm}[1]{\lVert#1\rVert}
\newcommand{\diff}{\mathop{}\!\mathrm{d}}

\DeclareMathOperator*{\argmin}{arg\,min}
\DeclareMathOperator{\supp}{supp}
\DeclareMathOperator{\trace}{tr}
\DeclareMathOperator{\spn}{span}
\DeclareMathOperator{\diag}{diag}

\theoremstyle{plain}
\newtheorem{theorem}{Theorem}[section]
\newtheorem{lemma}{Lemma}[section]

\theoremstyle{definition}
\newtheorem{definition}{Definition}[section]
\newtheorem{example}{Example}[section]

\theoremstyle{remark}
\newtheorem{remark}{Remark}[section]

\begin{document}

\title{Chebyshev smoothing with adaptive block-FSAI preconditioners for the multilevel solution of higher-order problems}
\date{}

\author[1,2]{Pablo Jim\'{e}nez Recio}
\author[1,2]{Marc Alexander Schweitzer}

\affil[1]{Institute for Numerical Simulation, Universit\"{a}t Bonn, 53111 Bonn, Germany}
\affil[2]{Fraunhofer SCAI, 53757 Sankt Augustin, Germany}

\renewcommand\Affilfont{\itshape\small}

\maketitle

\begin{abstract}
	In this paper, we assess the performance of adaptive and nested factorized sparse approximate inverses as smoothers in multilevel V-cycles, when smoothing is performed following the Chebyshev iteration of the fourth kind. For our test problems, we rely on the partition of unity method to discretize the biharmonic and triharmonic equations in a multilevel manner. Inspired by existing algorithms, we introduce a new adaptive algorithm for the construction of sparse approximate inverses, based on the block structure of matrices arising in the partition of unity method. Additionally, we also present a new (and arguably simpler) formulation of the Chebyshev iteration of the fourth kind.
\end{abstract}

\section{Introduction}\label{sec:intro}

The use of the Chebyshev iteration for concatenating smoothing steps in multigrid/multilevel iterations has been redefined by the significant work of Lottes \cite{Lottes2023}, which has allowed dispensing with the estimate of the lower end of the fine spectrum of the preconditioned matrix. The new Chebyshev iteration of the fourth kind was quickly adopted in PETSc and the deal.II library \cite{Arndt2023}. In this work, we combine this new Chebyshev iteration with underlying Factorized Sparse Approximate Inverse (FSAI) preconditioners, as developed in recent years by Janna et al. \cite{Janna2010, Janna2011, Janna2015}, and which have already been used as smoothers in multilevel approaches \cite{PaludettoMagri2019, Janna2025}. We will put these preconditioners into practice for a variety of problems generated with the Partition of Unity Method (PUM) of Schweitzer and Griebel \cite{Schweitzer2003, Griebel2002}.

The paper is organized as follows. In \cref{sec:fsai} we introduce the block-FSAI preconditioner, where our focus on block matrices follows from the natural block structure of matrices arising from the PUM. We pay special attention to developing our own adaptive algorithm for the construction of a suitable sparsity pattern for PUM block matrices, based on the one from \cite{Janna2011}, and also describe the possibility of nesting FSAI preconditioners.

In \cref{sec:chebyshev_iteration} we present the Chebyshev iteration, explain its \emph{raison d'être} and delve into its use as a way to optimally concatenate smoothing steps in multilevel iterations, paying special attention to the work of Lottes \cite{Lottes2023}. Given the simplicity of the new Chebyshev iteration of the fourth kind (which we simplify even further), we choose it for our work instead of the optimized (but more involved) version also introduced by Lottes.

In \cref{sec:pum} we provide a brief description of the PUM, whose discrete function spaces possess a natural geometric multilevel structure, as well as basis functions of any desired regularity.

Finally, in \cref{sec:experiments} we present some numerical experiments to attest the effectiveness and efficiency of the (adaptive and nested) block-FSAI preconditioners, which we embed within a Chebyshev iteration. The PUM allows us to test our multilevel solvers on higher-order problems such as the biharmonic and triharmonic equations, which provide a challenging scenario, since the usual Jacobi and Gauss-Seidel iterations are not particularly effective smoothers in this case, or at least not for PUM discretizations. We note that the multilevel solution of the biharmonic equation has already been investigated by other authors, see e.g.\ \cite{Sogn2019} and the references therein.

\section{The FSAI preconditioner}\label{sec:fsai}

For an invertible matrix \(A\in \mathbb{R}^{N\times N}\) and some predefined sparsity pattern \(\mathcal{P} \subset \{1, \dots, N\}^2\), a \emph{sparse approximate inverse} (SAI) of \(A\) based on \(\mathcal{P}\) is some matrix \(M\in \mathbb{R}^{N\times N}\) with sparsity pattern \(\mathcal{P}(M) \coloneqq \{(i,j) \,\colon\, M_{ij} \neq 0\} \subset \mathcal{P}\) which approximates \(A^{-1}\), usually by solving
\begin{equation*}
	M \coloneqq \argmin_{\mathcal{P}(\tilde{M}) \subset \mathcal{P}} \norm{I - \tilde{M} A}_F,
\end{equation*}
with \(\norm{\cdot}_F\) the Frobenius norm, defined by \(\norm{B}_F^2 = \trace(B^{\top} B)\). This constitutes the basis of the SAI preconditioner \cite{Sedlacek2012}.

Note however that, even if the matrix \(A\) is symmetric positive definite (s.p.d.), the resulting matrix \(M\) will not be s.p.d., and thus the preconditioned matrix will also not be s.p.d. The Factorized Sparse Approximate Inverse (FSAI) preconditioner overcomes this issue, and (following the construction of Janna et al. \cite{Janna2015}) it does so by constructing a sparse approximation \(G \approx L^{-1}\) to the inverse of the (unknown) triangular Cholesky factor \(L\) of \(A = L L^{\top}\), leading to \(M \coloneqq G^{\top} G \approx L^{-\top} L^{-1} = A^{-1}\), with the preconditioned matrix being \(G A G^{\top}\).

Let us first introduce the notation that we will use for working with block-structured matrices as those arising in the PUM.

\begin{definition}\label{def:block_structure}
	Let \(n \in \mathbb{N}\) and \(\mathfrak{B} \coloneqq \{m_i\}_{i=1}^n\) a set of \(n\) positive integers with \(N = \sum_{i=1}^n m_i\). By saying that a matrix \(A\in \mathbb{R}^{N \times N}\) has the block structure \(\mathfrak{B}\), we allow ourselves to interpret it as
	\begin{align*}
		A &= (A_{ij}), \quad i,j \in \{1,\dots,n\}, \\
		A_{ij} &= (a_{kl}) \in \mathbb{R}^{m_i\times m_j}.
	\end{align*}
	We further denote
	\begin{equation*}
		\mathcal{P}(A) \coloneqq \{(i, j)\, \colon \, A_{ij} \neq \mathbb{O}_{m_i \times m_j}\} \subset \{1,\dots,n\}^2
	\end{equation*}
	the sparsity pattern of such a matrix, and call a matrix block-diagonal if
	\begin{equation*}
		\mathcal{P}(A) \subset \{(i, i)\,\colon\, i \in \{1,\dots,n\}\}.
	\end{equation*}	
	Finally, we denote \(\diag_{\mathfrak{B}}(A)\) the block-diagonal component of a matrix \(A\), i.e.\ the matrix with diagonal blocks \(A_{ii}\) and \(\mathbb{O}_{m_i \times m_j}\) off-diagonal blocks.
\end{definition}

For any block matrix \(B\) as in \cref{def:block_structure} and sets of indices \(\mathcal{I},\mathcal{J}\subset \{1,\dots,n\}\), we will denote \(B[\mathcal{I},\mathcal{J}]\) the restriction of \(B\) to row indices in \(\mathcal{I}\) and column indices in \(\mathcal{J}\). We may now define the block-FSAI preconditioner.

\begin{definition}\label{def:fsai_matrices}
	Let \(A\) be a s.p.d. matrix with block structure \(\mathfrak{B} \coloneqq \{m_i\}_{i=1}^n\) as in \cref{def:block_structure}. Let \(\mathcal{P} \subset \{1,\dots,n\}^2\) be some lower triangular sparsity pattern containing the diagonal, i.e.
	\begin{equation}\label{eq:lower_triangular_pattern}
		(i, i) \in \mathcal{P} \quad \forall i \in \{1, \dots, n\},\qquad \text{and} \qquad i \geq j \quad \forall (i, j) \in \mathcal{P},
	\end{equation}
	Let us denote its rows, with and without the corresponding row index, by
	\begin{equation*}
		\mathcal{P}_i \coloneqq \{j \,\colon\, (i,j)\in \mathcal{P}\}, \quad \tilde{\mathcal{P}}_i \coloneqq \mathcal{P}_i \setminus \{i\}, \qquad i \in \{1,\dots,n\}.
	\end{equation*}
	We define the \emph{block-FSAI matrix for \(A\) based on \(\mathcal{P}\)} as the matrix \(F\) with block structure \(\mathfrak{B}\), sparsity pattern \(\mathcal{P}(F)\subset\mathcal{P}\) and non-zero block-entries given by
	\begin{equation}\label{eq:fsai_f_definition}
		F_{ii} = \mathbb{I}_{m_i \times m_i}, \quad F[\{i\}, \, \tilde{\mathcal{P}}_i] = - A[\{i\},\, \tilde{\mathcal{P}}_i]\, A[\tilde{\mathcal{P}}_i, \tilde{\mathcal{P}}_i]^{-1}, \qquad i \in \{1,\dots,n\}.
	\end{equation}
	
	The \emph{block-FSAI preconditioning matrix for \(A\) based on \(\mathcal{P}\)} is then \(M \coloneqq F^{\top} S^{-1} F\), where \(S\) is the block-diagonal matrix (also with block structure \(\mathfrak{B}\)) given by
	\begin{equation}\label{eq:schur_complement_fsai}
		S_{ii} \coloneqq A_{ii} - A[\{i\},\, \tilde{\mathcal{P}}_i] A[\tilde{\mathcal{P}}_i, \tilde{\mathcal{P}}_i]^{-1} A[\tilde{\mathcal{P}}_i, \{i\}], \qquad i \in \{1,\dots,n\}.
	\end{equation}
	
	We may also write \(M = G^{\top} G\) for \(G = S^{-1/2} F\), with the preconditioned matrix then being \(G A G^{\top}\).
\end{definition}

\begin{remark}
	From now on, unless explicitly stated otherwise, every reference to a ``lower triangular sparsity pattern'' will mean also including the diagonal, i.e.\ a sparsity pattern fulfilling both conditions in \eqref{eq:lower_triangular_pattern}.
\end{remark}

\begin{lemma}\label{thm:fsai_block_diagonal}
	Let \(A\) be a s.p.d. matrix with block structure \(\mathfrak{B}\) and \(F\) be defined as in \eqref{eq:fsai_f_definition} for some lower triangular sparsity pattern \(\mathcal{P}\). Let \(S\) be the block-diagonal matrix with blocks given by \eqref{eq:schur_complement_fsai}. It holds that
	\begin{equation*}
		\diag_{\mathfrak{B}}(FA) = S = \diag_{\mathfrak{B}}(FAF^{\top}),
	\end{equation*}
	and in particular it follows that the block-diagonal of the block-FSAI preconditioned matrix, \(\diag_{\mathfrak{B}}(GAG^{\top})\), is the identity matrix.
\end{lemma}
\begin{proof}
	The first identity follows directly from \eqref{eq:fsai_f_definition} and \eqref{eq:schur_complement_fsai}. Indeed,
	\begin{equation*}
		(FA)_{ii} = F[\{i\}, \, \mathcal{P}_i] \, A[\mathcal{P}_i,\, \{i\}] = S_{ii}, \quad \forall i \in \{1,\dots,n\}.
	\end{equation*}
	The second one can also be easily checked relying on the first one and writing
	\begin{multline*}
		(F A F^{\top})_{ii} = \begin{pmatrix}
			I_{ii} & F[\{i\}, \, \tilde{\mathcal{P}}_i]
		\end{pmatrix}
		\begin{pmatrix}
			A_{ii} & A[\{i\},\, \tilde{\mathcal{P}}_i] \\
			A[\tilde{\mathcal{P}}_i, \{i\}] & A[\tilde{\mathcal{P}}_i, \tilde{\mathcal{P}}_i] 
		\end{pmatrix}
		\begin{pmatrix}
			I_{ii} \\ F^{\top}[\tilde{\mathcal{P}}_i, \, \{i\}]
		\end{pmatrix} =\\
		= \begin{pmatrix}
			I_{ii} & F[\{i\}, \, \tilde{\mathcal{P}}_i]
		\end{pmatrix}
		\begin{pmatrix}
			S_{ii} \\ 0
		\end{pmatrix} = S_{ii}.
	\end{multline*}
	The final statement follows trivially for \(G = S^{-1/2} F\).
\end{proof}

\begin{remark}
	For the simplest choice of a diagonal sparsity pattern
	\begin{equation*}
		\mathcal{P} = \{(i, i) \,\colon\, i \in\{1, \dots, n\}\},
	\end{equation*}	
	the resulting matrix \(F\) is the identity matrix and \(S = \diag_{\mathfrak{B}}(A)\), so the block-FSAI preconditioner reduces to a block-Jacobi preconditioner.
\end{remark}

Let us now introduce an important concept which, as we will show, is deeply connected to the FSAI preconditioner.

\begin{definition}\label{def:kaporin_number}
	For a s.p.d. matrix \(B\in\mathbb{R}^{N\times N}\), the \emph{Kaporin number} is defined as
	\begin{equation*}
		\beta(B) \coloneqq \frac{\trace(B)}{N \,\det(B)^{1/N}},
	\end{equation*}
	that is, the ratio between the algebraic and the geometric means of the eigenvalues of \(B\).
\end{definition}

\begin{remark}\label{rm:kaporin_properties}
	Some trivial properties of the Kaporin number \(\beta(\cdot)\) are that \(\beta(B) \geq 1\), and that \(\beta(B) = 1\) if and only if all eigenvalues of \(B\) coincide. Additionally, from this last property it follows that, for any full-rank matrix \(C\), \(\beta(C B C^{\top}) = 1\) if and only if \(C^{\top} C = \alpha B^{-1}\) for some \(\alpha > 0\).
\end{remark}

We now reproduce a result from Janna et al. \cite{Janna2011}, which will be useful for the subsequent theorem.

\begin{lemma}\label{thm:optimal_diagonal_scaling}
	Let \(B \in \mathbb{R}^{N\times N}\) be a s.p.d. matrix with block structure \(\mathfrak{B}\). Then \(\beta(J B J^{\top})\) is minimized over full-rank block-diagonal matrices \(J\) if and only if there exists \(\alpha > 0\) such that \((J^{\top} J)_{kk} = \alpha (B_{kk})^{-1}\).
\end{lemma}
\begin{proof}
	Let \(D = \diag_{\mathfrak{B}}(B)\), so that in particular, \(\trace(J B J^{\top}) = \trace(J D J^{\top})\) and also \(\trace(D^{-1} B) = N\). Since \(\det(J B J^{\top}) = \det(J D J^{\top}) \det(D^{-1} B)\), it can easily be shown that
	\begin{multline*}
		\beta(J B J^{\top}) = \frac{\trace(J B J^{\top})}{N \det(J B J^{\top})^{1/N}} = \frac{\trace(J D J^{\top})}{N \det(J D J^{\top})^{1/N}} \, \frac{N}{N \det(D^{-1} B)^{1/N}} =\\
		= \beta(J D J^{\top}) \, \beta(D^{-1/2} B D^{-1/2}).
	\end{multline*}
	The problem has thus been reduced to minimizing \(\beta(J D J^{\top})\), and the claim follows by noting that \(\beta(J D J^{\top}) \geq 1\) and that the value of 1 is attained if and only if \(J^{\top}J = \alpha D^{-1}\) for some \(\alpha > 0\) (recall \cref{rm:kaporin_properties}).
\end{proof}

We may now introduce the well-known optimality property of the FSAI preconditioner with respect to the Kaporin number (see the work of Kaporin himself \cite{Kaporin1994}).
\begin{theorem}[FSAI optimality]\label{thm:fsai_kaporin}
	Let \(A\) be a s.p.d. matrix with block structure \(\mathfrak{B}\), \(\mathcal{P}\) be some lower triangular sparsity pattern, and \(F, S\) be the block-FSAI matrices based on \(\mathcal{P}\) as in \cref{def:fsai_matrices}.
	
	Then \(G \coloneqq S^{-1/2} F\) minimizes the Kaporin number of the preconditioned matrix \(\beta(\tilde{G} A \tilde{G}^{\top})\) over all full-rank matrices \(\tilde{G}\) with sparsity pattern \(\mathcal{P}\).
\end{theorem}
\begin{proof}
	We proceed by characterizing the minimizers of \(\beta(\tilde{G} A \tilde{G}^{\top})\) over the considered matrices.
	
	First of all, any full-rank matrix with (lower triangular) sparsity pattern \(\mathcal{P}\) can be written as \(\tilde{G} = J(I + \tilde{F})\), for some non-singular block-diagonal matrix \(J\) and some (also lower triangular) \(\tilde{F}\) having \(\mathbb{O}_{m_i \times m_i}\) as diagonal blocks (i.e.\ being strictly lower triangular). As a result, and since \(\det(I + \tilde{F}) = 1\), we can write
	\begin{equation*}
		\beta(\tilde{G} A \tilde{G}^{\top}) = \frac{\trace\left( J (I + \tilde{F}) A (I + \tilde{F}^{\top}) J^{\top} \right)}{N \left[ \det(A)\det(J^{\top} J) \right]^{1/N}},
	\end{equation*}
	where only the numerator depends on \(\tilde{F}\). Moreover,
	\begin{equation*}
		\trace\left( J (I + \tilde{F}) A (I + \tilde{F}^{\top}) J^{\top} \right) = \sum_{k=0}^n \trace\left( ((I + \tilde{F}) A (I + \tilde{F})^{\top})_{kk}  (J^{\top}J)_{kk}\right),
	\end{equation*}
	so when minimizing \(\beta(\tilde{G} A \tilde{G}^{\top})\) with respect to \(\tilde{F}\), the non-zero entries in the \(k\)th row of \(\tilde{F}\), i.e.\ \(\tilde{F}[\{k\}, \tilde{\mathcal{P}}_k]\), must minimize
	\begin{multline*}
		\trace\left( ((I + \tilde{F}) A (I + \tilde{F})^{\top})_{kk}  (J^{\top}J)_{kk}\right) = \trace\left( A_{kk} (J^{\top}J)_{kk} \right) +\\
		+ \trace\left( A[\{k\}, \tilde{\mathcal{P}}_k] \tilde{F}^{\top}[\tilde{\mathcal{P}}_k, \{k\}] (J^{\top}J)_{kk} \right) +\\
		+ \trace\left( \tilde{F}[\{k\}, \tilde{\mathcal{P}}_k] A[\tilde{\mathcal{P}}_k, \{k\}] (J^{\top}J)_{kk} \right) +\\
        + \trace\left( \tilde{F}[\{k\}, \tilde{\mathcal{P}}_k] A[\tilde{\mathcal{P}}_k, \tilde{\mathcal{P}}_k] \tilde{F}^{\top}[\tilde{\mathcal{P}}_k, \{k\}] (J^{\top}J)_{kk} \right),
	\end{multline*}
	or equivalently
	\begin{equation*}
		2 \, \trace\left( A[\tilde{\mathcal{P}}_k, \{k\}] (J^{\top}J)_{kk} \right) + \trace\left( \tilde{F}[\{k\}, \tilde{\mathcal{P}}_k] A[\tilde{\mathcal{P}}_k, \tilde{\mathcal{P}}_k] \tilde{F}^{\top}[\tilde{\mathcal{P}}_k, \{k\}] (J^{\top}J)_{kk} \right).
	\end{equation*}
	Differentiating the above expression with respect to \(\tilde{F}[\{k\}, \tilde{\mathcal{P}}_k]\) (cf. \cite[Section~2.5]{Petersen2012}) and setting the result to 0 yields
	\begin{equation*}
		2 A[\tilde{\mathcal{P}}_k, \{k\}] (J^{\top}J)_{kk} + 2 A[\tilde{\mathcal{P}}_k, \tilde{\mathcal{P}}_k] \tilde{F}^{\top}[\tilde{\mathcal{P}}_k, \{k\}] (J^{\top}J)_{kk} = \mathbb{O}_{m_k \times m_k},
	\end{equation*}
	and since \((J^{\top}J)_{kk}\) is a s.p.d. block, we arrive at
	\begin{equation*}
		\tilde{F}[\{k\}, \tilde{\mathcal{P}}_k] = - A[\{k\}, \tilde{\mathcal{P}}_k] A[\tilde{\mathcal{P}}_k, \tilde{\mathcal{P}}_k]^{-1},
	\end{equation*}
	which means that \(I + \tilde{F}\) coincides with the FSAI matrix \(F\) from \eqref{eq:fsai_f_definition}, independently of \(J\).
	
	Having fixed \(\tilde{F}\), it remains to minimize \(\beta\left( J (F A F^{\top})  J^{\top}\right)\) with respect to \(J\). This last step is quite simple, since from \cref{thm:optimal_diagonal_scaling} it follows that
	\(J^{\top}J = \alpha \, \diag_{\mathfrak{B}}(F A F^{\top})\) for some \(\alpha > 0\), which, recalling \cref{thm:fsai_block_diagonal}, is fulfilled by \(J = S^{-1/2}\) with \(\alpha = 1\).
\end{proof}

\subsection{Adaptive block-FSAI}\label{sec:adaptive_fsai}

It remains to address the question about how to properly choose the triangular sparsity pattern \(\mathcal{P}\) for the computation of \(F\). The simplest case is to fix it a priori, e.g.\ as the triangular pattern extracted from the original matrix \(A\), or from some power of it: \(A^2\), \(A^3\), etc.

That approach, however, does not rely on the actual entries of \(A\), but only on its structure. A more involved approach is to adaptively construct a sparsity pattern relying specifically on the entries of \(A\). Following the work Janna and Ferronato \cite{Janna2011}, we will develop an adaptive algorithm with the aim of minimizing the Kaporin number of the preconditioned matrix, \(\beta(G A G^{\top})\). It will rely on the following result, also taken from \cite{Janna2011}.

\begin{lemma}
	Let \(A \in \mathbb{R}^{N\times N}\) be a s.p.d. matrix with block structure \(\mathfrak{B} = \{m_i\}_{i=1}^n\), \(\mathcal{P}\) be some lower triangular sparsity pattern, and \(F\) and \(S\) be the FSAI matrices based on \(\mathcal{P}\) as in \cref{def:fsai_matrices}. Then, for \(G \coloneqq S^{-1/2} F\), it holds that
	\begin{equation*}
		\beta(G A G^{\top}) = \left( \frac{\det(\diag_{\mathfrak{B}}(F A F^{\top}))}{\det(A)} \right)^{1/N} = \left( \frac{\prod_{i=1}^n\det((F A F^{\top})_{ii})}{\det(A)} \right)^{1/N}.
	\end{equation*}
\end{lemma}

\begin{proof}
	From \cref{thm:fsai_block_diagonal} it follows that \(\trace(G A G^{\top}) = N\), and further noting that \(\det(F) = 1\),
	\begin{equation*}
		\beta(G A G^{\top}) = \frac{\trace(G A G^{\top})}{N \,\det(G A G^{\top})^{1/N}} = \left(\frac{\det(S)}{\det(A)}\right)^{1/N} = \left( \frac{\prod_{i=1}^n\det((F A F^{\top})_{ii})}{\det(A)} \right)^{1/N}.
	\end{equation*}
\end{proof}

Therefore, when enlarging the \(k\)th row of the sparsity pattern, \(\mathcal{P}_k\), we will aim at minimizing \(\det((F A F^{\top})_{kk})\). Let us now recall \cref{thm:fsai_block_diagonal} in the form
\begin{equation*}
	(FA)_{kk} = (FAF^{\top})_{kk} = A_{kk} - A[\{k\},\, \tilde{\mathcal{P}}_k] A[\tilde{\mathcal{P}}_k,\, \tilde{\mathcal{P}}_k]^{-1} A[\tilde{\mathcal{P}}_k,\, \{k\}],
\end{equation*}
and additionally note that, by construction of \(F\),
\begin{equation}\label{eq:fsai_helper}
	(FA)[\{k\},\, \tilde{\mathcal{P}}_k] = A[\{k\},\, \tilde{\mathcal{P}}_k] - A[\{k\},\, \tilde{\mathcal{P}}_k]\, A[\tilde{\mathcal{P}}_k,\, \tilde{\mathcal{P}}_k]^{-1} \, A[\tilde{\mathcal{P}}_k,\, \tilde{\mathcal{P}}_k] \equiv 0,
\end{equation}
so that \(\mathcal{P}(FA) \cap \mathcal{P} = \{(k, k)\,\colon\, k \in \{1,\dots,n\}\}\). We now consider how
\begin{equation*}
	\Delta = \det((F A F^{\top})_{kk})
\end{equation*}
would be affected if we were to add a single index \(c \in \{0,\dots,k-1\} \setminus \mathcal{P}_k\) to the \(k\)th row of the sparsity pattern.
\begin{lemma}\label{thm:fsai_adaptivity}
	Let \(A\) be a s.p.d. matrix with block structure \(\mathfrak{B} = \{m_i\}_{i=1}^n\) and \(F\) be the FSAI matrix for a given lower triangular sparsity pattern \(\mathcal{P}\). Let \(k \in \{1,\dots,n\}\) be some row index, and \(c \in \{0,\dots,k-1\} \setminus \mathcal{P}_k\) some admissible column index. If we denote \(F^{(c)}\) the FSAI matrix that would result from the expanded sparsity pattern \(\mathcal{P} \cup \{(k,c)\}\), it holds that
	\begin{equation*}
		(F^{(c)} A F^{(c)\top})_{kk} = (F A F^{\top})_{kk} - (FA)_{kc} W_c^{-1} (FA)^{\top}_{kc},
	\end{equation*}
	where
	\begin{equation*}
		W_c \coloneqq A_{cc} - A[\{c\},\, \tilde{\mathcal{P}}_k] \, A[\tilde{\mathcal{P}}_k,\, \tilde{\mathcal{P}}_k]^{-1}\, A[\tilde{\mathcal{P}}_k,\, \{c\}].
	\end{equation*}
	Additionally, if we denote \(\Delta = \det((F A F^{\top})_{kk})\) and \(\Delta^{(c)} = \det((F^{(c)} A F^{(c)\top})_{kk})\), and let \(H = FA\), the following relationship holds
	\begin{equation*}
		\frac{\Delta^{(c)}}{\Delta} = \frac{\det\left(W_c - H_{kc}^{\top} H_{kk}^{-1} H_{kc}\right)}{\det(W_c)}.
	\end{equation*}
\end{lemma}
\begin{proof}
	For the first claim, we merely describe the otherwise tedious process for arriving at the desired expression. First we write
	\begin{multline*}
		(F^{(c)} A F^{(c)\top})_{kk} =\\
        = A_{kk} -
		\begin{pmatrix}
			A_{kc} & A[\{k\},\, \tilde{\mathcal{P}}_k]
		\end{pmatrix} \begin{pmatrix}
			A_{cc} & A[\{c\},\, \tilde{\mathcal{P}}_k] \\
			A[\tilde{\mathcal{P}}_k,\, \{c\}] & A[\tilde{\mathcal{P}}_k,\, \tilde{\mathcal{P}}_k]
		\end{pmatrix}^{-1} \begin{pmatrix}
			A_{ck} \\ A[\tilde{\mathcal{P}}_k, \{k\}]
		\end{pmatrix},
	\end{multline*}
	and then we expand the inverse in terms of the Schur complement matrix \(W_c\). The desired expression can be obtained by properly gathering all terms, while also taking into account that
	\begin{equation*}
		(FA)_{kc} = A_{kc} - A[\{k\},\, \tilde{\mathcal{P}}_k] \, A[\tilde{\mathcal{P}}_k,\, \tilde{\mathcal{P}}_k]^{-1} A[\tilde{\mathcal{P}}_k, \{c\}]\,.
	\end{equation*}
	The result for \(\Delta^{(c)} / \Delta\) follows by the relation of determinants of Schur complement matrices with those of the original matrix, and \(H_{kk} = (FA)_{kk} = (FAF^{\top})_{kk}\).
\end{proof}

This last result thus provides a way to compute \(\Delta^{(c)} / \Delta\) without constructing the new \(k\)th row of the matrix \(F^{(c)}\). One further trivial consequence is that we may ignore column indices \(c\) for which \((FA)_{kc} \equiv 0\). At this point, we are able to present our simple adaptive block-FSAI algorithm in \cref{alg:adaptive_fsai}.

\begin{algorithm}[!htb]
	\caption{Adaptive block-FSAI}\label{alg:adaptive_fsai}
	\begin{algorithmic}[1]
		\Require{A s.p.d. matrix \(A\) with the block structure \(\mathfrak{B} = \{m_i\}_{i=0}^{n}\), and an initial lower triangular sparsity pattern \(\mathcal{P}_0 \subset \{1,\dots,n\}^2\) on said block structure, containing all diagonal pairs of indices. The number of adaptive steps \(t_{max} \geq 1\) and a threshold parameter \(\tau \in (0, 1]\).}
		
		\vspace{1em}
		
		\State{Initialize \(\mathcal{P} = \mathcal{P}_0\).}
		\State{Initialize \(\mathcal{R} = \{1,\dots,n\}\) to be the set of row indices which may still be adapted.}
		
		\vspace{1em}

		\For{\(t=1,\dots,t_{max}\)}
		
		\State{Compute the FSAI matrix \(F\) based on \(\mathcal{P}\), according to \eqref{eq:fsai_f_definition}.}
		
		\Comment{This requires computing \(A[\tilde{\mathcal{P}}_k,\, \tilde{\mathcal{P}}_k]^{-1}\) for every \(k\in \mathcal{R}\), which we may store.}
		
		\vspace{1em}
		
		\State{Compute \(H = F A\).}
		
		\Comment{For the last two steps, we may skip finished rows \(k \not\in \mathcal{R}\).}
		
		\vspace{1em}
		
		\For{\(k \in\mathcal{R}\)}
		
		\State{Let \(\mathcal{C} \coloneqq \{c \in \{0,\dots,k-1\} \, \colon \, H_{kc} \not\equiv 0\}\) be the set of admissible column indices.}
		
		\Comment{From \eqref{eq:fsai_helper}, \(\mathcal{C} \cap \mathcal{P}_k = \emptyset\), i.e.\ no index already in the row pattern is admissible.}
		
		\vspace{1em}
		
		\If{\(\mathcal{C} = \emptyset\)}
		\State{Remove \(k\) from \(\mathcal{R}\) and continue.}
		\EndIf
		
		\vspace{1em}
		
		\State{Let \(W_c \coloneqq A_{cc} - A[\{c\},\, \tilde{\mathcal{P}}_k] \, A[\tilde{\mathcal{P}}_k,\, \tilde{\mathcal{P}}_k]^{-1}\, A[\tilde{\mathcal{P}}_k,\, \{c\}]\) for \(c \in \mathcal{C}\), and find
		\begin{equation*}
			c^{\star} \in \argmin_{c \in \mathcal{C}} \rho_c, \quad \rho_c \coloneqq \frac{\det(W_c - H_{kc}^{\top} H_{kk}^{-1} H_{kc})}{\det(W_c)}.
		\end{equation*}}

		\Comment{We may reuse \(A[\tilde{\mathcal{P}}_k,\, \tilde{\mathcal{P}}_k]^{-1}\), computed at the latest assembly of \(F\).}
		
		\vspace{1em}
		
		\If{\(\rho_{c^{\star}} < \tau\)}
		\State{\(\mathcal{P} \gets \mathcal{P} \cup \{(k, \, c^{\star})\}\).}
		\Else
		\State{\(\mathcal{R} \gets \mathcal{R} \setminus \{k\}\).}
		\EndIf
		
		\vspace{1em}
		
		\EndFor
		\EndFor
		
		\vspace{1em}
		
		\State{Compute the FSAI matrix \(F\) based on the current state of \(\mathcal{P}\), according to \eqref{eq:fsai_f_definition}.}
		
		\Comment{As above, it is unnecessary to recompute finished rows \(k \not\in \mathcal{R}\).}

        \vspace{1em}

		\Return{\(F\)}
	\end{algorithmic}
\end{algorithm}

Compared to the algorithm from \cite{Janna2011}, ours does not use gradient descent on the minimization target (since the problem of constructing the sparsity pattern is purely discrete), relies only on a single threshold parameter and a single number of adaptive steps, and does not allow backtracking by removing pairs of indices from the sparsity pattern. Hence our claim for ``simplicity'', even at the cost of flexibility. Note that, since we add at most one new column index per row at each adaptive step, the \(t_{max}\) parameter determines the maximum number of off-diagonal blocks in any row of the resulting \(F\).

For a similar analysis to ours above we refer to M. Sedlacek's PhD thesis \cite[Appendix~C]{Sedlacek2012}.

\subsection{Nested FSAI}\label{sec:recursive_fsai}

\cref{thm:fsai_block_diagonal} allows us to reinterpret FSAI preconditioning as applying the basis transformation \(F\), followed by a single step of the block-Jacobi iteration for the transformed matrix \(F A F^{\top}\), and finally a basis transformation back to the original space \(F^{\top}\). In other words, the action of the preconditioner \(F^{\top} S^{-1} F\) can be understood as
\begin{equation*}
	r \longmapsto v = Fr \longmapsto w = \diag_{\mathfrak{B}}(F A F^{\top})^{-1} v \longmapsto p = F^{\top} w.
\end{equation*}
This interpretation suggests one further twist: instead of the block-Jacobi step for the FSAI-transformed matrix \(F A F^{\top}\), we can apply any other (preferably s.p.d.) preconditioner, and in particular, an additional step of FSAI, this one constructed for the matrix \(F A F^{\top}\) (contrary to the outer one, which is constructed simply for \(A\)). In the literature, this has been called first \emph{recurrent} FSAI \cite{Janna2015}, and recently \emph{nested} FSAI \cite{Janna2025}. Ideally, one can concatenate any number of embeddings, i.e.
\begin{equation*}
	A_0 \coloneqq A, \quad A_k = F_{k-1} A_{k-1} F_{k-1}^{\top}, \quad k = 1, \dots, n_{\ell},
\end{equation*}
where \(F_k\) is the FSAI matrix for \(A_k\) (based on some sparsity pattern). At the final level \(n_{\ell}\), a standard block-Jacobi step is performed. The matrix \(M\) of the resulting preconditioner can be written as
\begin{equation*}
	M = \hat{F}^{\top} \hat{S}^{-1} \hat{F}, \quad \hat{F} \coloneqq F_{n_{\ell}} F_{n_{\ell}-1} \cdots F_0, \quad \hat{S} \coloneqq \diag_{\mathfrak{B}}(A_{n_{\ell}}) = \diag_{\mathfrak{B}}(\hat{F} A \hat{F}^{\top}).
\end{equation*}

Nesting thus allows us to increase the density of the preconditioning matrix \(M\) while avoiding the computational cost of doing it with a single FSAI matrix \(F\), which would require a sufficiently dense sparsity pattern. A higher density of \(M\) is achieved by combining multiple \(F\) matrices of limited density, each of which can be computed adaptively. The cost is thus transferred to computing each triple matrix product \(A_{k+1} = F_k A_k F_k^{\top}\), which is required to construct the subsequent \(F_{k+1}\).

\section{The Chebyshev iteration}\label{sec:chebyshev_iteration}

Let us now look at a completely different preconditioner, based on a polynomial iteration. Polynomial iterations are iterative methods for the approximation of a solution to a linear system \(Ax = b\), such that the error \(e_k \coloneqq x - x_k\) evolves as
\begin{equation*}
	e_k = p_k(A) e_0, \quad k \geq 1,
\end{equation*}
where \(e_0\) is the error of the initial solution \(x_0\), and \(p_k\) is some polynomial of degree \(k\) with \(p_k(0) = 1\). This is equivalent to having
\begin{equation}\label{eq:polynomial_iteration_x}
	x_k = x_0 + q_{k-1}(A) (b - A x_0), \quad k \geq 1,
\end{equation}
where \(p_k(t) = 1 - t q_{k-1}(t)\). Additionally, \(p_k(A)\) controls the residual evolution, i.e.\ \(r_k = p_k(A) r_0\), where \(r_k = b - A x_k\). Under the assumption that \(A\) is a s.p.d. matrix, we have that,
\begin{equation*}
	\frac{\norm{e_k}_2}{\norm{e_0}_2} \leq \max_{\lambda\in \sigma(A)} |p_k(\lambda)| \leq \max_{\lambda\in [\alpha,\, \beta]} |p_k(\lambda)|,
\end{equation*}
for an interval \([\alpha,\, \beta] \supset \sigma(A)\). The unique solution to the minimax problem
\begin{equation}\label{eq:chebyshev_minimization}
	\min_{\substack{p_k \in \mathbb{P}_k\\ p_k(0) = 1}} \max_{\lambda\in [\alpha,\, \beta]} |p_k(\lambda)|, \quad \beta > \alpha > 0,
\end{equation}
where \(\mathbb{P}_k\) is the set of polynomials of degree \(k\), is the reparametrized and rescaled Chebyshev polynomial of degree \(k\)
\begin{equation}\label{eq:chebyshev_1st_kind}
	\hat{T}_k(\zeta) \coloneqq \left[T_k\left(\frac{\alpha + \beta}{\alpha - \beta}\right)\right]^{-1} T_k\left(\frac{2\zeta - \alpha - \beta}{\beta - \alpha}\right),
\end{equation}
where \(T_k\) is the standard Chebyshev polynomial (of the first kind), based on the \([-1, 1]\) interval, and thus \(\hat{T}_k\) is reparametrized to the \([\alpha,\, \beta]\) interval and rescaled so that \(\hat{T}_k(0) = 1\). Let us introduce, in as simple a manner as possible, a minimal set of concepts to understand the proof.

\begin{definition}
	Let \([a, b] \subset \mathbb{R}\). We say that a set of functions \(\{g_i\}_{i=1}^n\), \(g_i \in C([a, b])\), is a \emph{Chebyshev system} (or a \emph{Haar system}) if every non-zero element in \(\spn\langle g_i \rangle_{i=1}^n\) has at most \(n - 1\) distinct roots in \([a, b]\).
\end{definition}

\begin{theorem}[Chebyshev's Alternation Theorem]\label{thm:alternation}
	Let \([a, b] \subset \mathbb{R}\), \(f \in C([a, b])\), and \(\{g_i\}_{i=1}^n\), \(g_i \in C([a, b])\), be a Chebyshev system. Then \(\hat{g} \in G \coloneqq \spn\langle g_i \rangle_{i=1}^n\) is a solution to the minimax problem
	\begin{equation*}
		\min_{g \in G} \max_{t\in [a,b]} |f(t) - g(t)|,
	\end{equation*}
	if and only if there exist \(n + 1\) points \(a \leq t_0 < \cdots < t_n \leq b\) satisfying
	\begin{align*}
		|r(t_i)| &= \max_{t\in [a,b]} |r(t)|, \quad i = 0,\dots,n,\\
		r(t_i) &= -r(t_{i-1}), \quad i = 1,\dots,n,
	\end{align*}
	for \(r(t) = f(t) - \hat{g}(t)\). The points \(\{t_i\}_{i=1}^n\) are called ``alternation points'' of \(r\) in \([a, b]\). Furthermore, if a minimizer exists, then it is unique.
\end{theorem}
\begin{proof}
	This is a combination of the Alternation Theorem and the Unicity Theorem in \cite{Cheney1998}. For a more recent version, we refer to \cite{Alimov2021}.
\end{proof}

\begin{lemma}
	For \(n > 0\), the Chebyshev polynomial \(T_n\), defined in the \([-1,1]\) interval by
	\begin{equation*}
		T_n(t) = \cos n\theta, \quad t = \cos \theta \in [-1,1],
	\end{equation*}	
	has the \(n + 1\) alternation points
	\begin{equation*}
		t_i = \cos \frac{i\pi}{n}, \quad i = 0,\dots,n.
	\end{equation*}
\end{lemma}
\begin{proof}
	Trivial given the trigonometric definition. For more details, we refer to \cite{Mason2002}.
\end{proof}

Usually, instead of for polynomials with a pointwise constraint, the minimax property of Chebyshev polynomials is stated for monic polynomials \cite{Mason2002, Mason2005}, which is easily proved using the Alternation Theorem to show that a minimizer of
\begin{equation*}
	\min_{q_{k-1} \in \mathbb{P}_{k-1}} \max_{t\in [-1, 1]} |t^k - q_{k-1}(t)|
\end{equation*}
satisfies \(t^k - q_{k-1}(t) \propto T_k(t)\), and hence the properly rescaled Chebyshev polynomial is the one least deviating from 0 in the \(L^{\infty}([-1,1])\)-norm among all monic polynomials. For the pointwise constraint \(p_k(\xi) = 1\) with \(\xi\in\mathbb{R}\), one simply needs to restate the minimax problem as
\begin{equation*}
	\min_{q_{k-1} \in \mathbb{P}_{k-1}} \max_{t\in [-1, 1]} |1 - (t - \xi)\, q_{k-1}(t)|,
\end{equation*}
which in this case introduces a caveat: in order to apply the Alternation Theorem, \(\{(t - \xi)\, t^i\}_{i=0}^{k-1}\) has to be a Chebyshev system in \(C([-1,1])\), which means that for any nonzero \(q \in \mathbb{P}_{k-1}\), \((t - \xi)\, q(t)\) may have at most \(k-1\) roots in \([-1, 1]\), which holds if and only if \(\xi \notin [-1, 1]\). This is the reason behind the requirement that \(\alpha > 0\) in \eqref{eq:chebyshev_minimization}.

Now that we have properly justified the use of Chebyshev polynomials for iterative methods, we need to find an appropriate construction of the iteration \eqref{eq:polynomial_iteration_x}. Following the recursive rule of Chebyshev polynomials
\begin{align*}
	T_0(t) &= 1,\\
	T_1(t) &= t,\\
	T_{n+1}(t) &= 2t\, T_n(t) - T_{n-1}(t), \quad n \geq 1,
\end{align*}
one can write the resulting linear iteration for the system \(Ax = b\) in a variety of different ways, as thoroughly discussed by Gutknecht and Röllin \cite{Gutknecht2002}. Based on their two-term iteration, we can write the Chebyshev iteration as in \cref{alg:chebyshev_iteration}.

\begin{algorithm}[!htb]
	\caption{Chebyshev iteration of degree \(k\)}\label{alg:chebyshev_iteration}
	\begin{algorithmic}[1]
		\Require{A linear system \(A x = b\) with a s.p.d. matrix \(A\). An initial tentative solution \(x_0\). An interval \([\alpha,\, \beta] \supset \sigma(A)\) with \(0 < \alpha < \beta\). A positive integer \(k \geq 1\).}
		
		\vspace{1em}
		
		\State{Rewrite \([\alpha,\, \beta] = [c-d,\, c+d]\), i.e.\ set}
		\begin{equation*}
			c = \frac{\alpha + \beta}{2}, \quad d = \frac{\beta - \alpha}{2}.
		\end{equation*}
		
		\State{Initialize \(r = b - Ax_0\), \(p = r\).}
		\State{Initialize $\omega = \frac{1}{c}$, $\psi = \frac{1}{2} ( d \omega)^2$.}
		
		\vspace{1em}
		
		\State{Initialize \(x = x_0 + \omega \, p\)}
		
		\vspace{1em}
		
		\For{\(i=2,\dots,k\)}
		
		\State{\(r \gets b - Ax\)}
		\State{\(p \gets r + \psi \, p\)}
		\State{\(\omega \gets \left(c - \frac{\psi}{\omega} \right)^{-1}\)}
		\State{\(\psi \gets \left( \frac{d \omega}{2} \right)^2\)}
		
		\vspace{1em}
		
		\State{\(x \gets x + \omega \, p\)}
		
		\EndFor
		
		\vspace{1em}
		
		\Return{\(x\)}
	\end{algorithmic}
\end{algorithm}

Additionally, the Chebyshev iteration naturally lends itself to preconditioning with a s.p.d. matrix \(M\), in which case the interval \([\alpha,\, \beta]\) should contain \(\sigma(MA)\) instead of \(\sigma(A)\). The inclusion of preconditioning in \cref{alg:chebyshev_iteration} is trivial: we simply replace \(r\) by \(Mr\) in the initialization as well as in the update of \(p\).

Although estimating the maximal eigenvalue of a s.p.d. matrix \(A\) (or a preconditioned \(MA\), with \(M\) also s.p.d.) is not a complicated task, the same cannot be said about estimating its minimal eigenvalue. Therefore, techniques have been devised to perform the Chebyshev iteration without actually estimating the minimal eigenvalue (see e.g.\ the LS-Chebyshev polynomial preconditioner from \cite{Ashby1992}). This issue will also play a role in the forthcoming subsection.

\subsection{Chebyshev smoothing}\label{sec:chebyshev_smoothing}

The Chebyshev iteration is also a Krylov space method, but contrary to the usual Krylov space methods (e.g.\ CG, GMRES), the absence of inner products in its iteration makes it particularly well-suited to be used not as a solver itself, but (for a fixed, predefined number of iterations) as a preconditioner within another solver (e.g.\ CG). Note that we presented \cref{alg:chebyshev_iteration} already in this spirit by fixing the number of iterations.

Furthermore, the Chebyshev iteration can be repurposed as a smoother for a multilevel iteration by targeting only the higher end of the spectrum of \(MA\), as initially performed by Adams et al. \cite{Adams2003}, and later by Baker et al. \cite{Baker2011}. This means that, instead of relying on an interval \([\alpha,\, \beta] \supset \sigma(MA)\), which would require estimates of both the maximal and minimal eigenvalues of \(MA\), one may cover only a right-subinterval of \(\sigma(A)\), e.g.\ in the form of \([\rho \, \lambda_{max}(MA), \, \lambda_{max}(MA)]\) for some \(\rho \in (0,1)\).

In this scenario, the usual smoothing stage consisting of \(k\) repeated applications of a preconditioner, i.e.\ \(S = (I - MA)^k\), also known as Richardson smoothing, can be replaced by a single application of the \(M\)-preconditioned Chebyshev iteration of degree \(k\), i.e.\ \(S = p_k(MA)\) (since it also involves \(k\) matrix-vector products with \(MA\)).

Recently, Lottes \cite{Lottes2023} has sharply reconsidered whether the minimax property \eqref{eq:chebyshev_minimization} satisfied by Chebyshev polynomials of the first kind is still meaningful when using the Chebyshev iteration to concatenate smoothing steps within a multilevel iteration. Following the splitting of the two-level error propagation leading to the Approximation Property and the Smoothing Property (see e.g.\ Equation (6.1.5) from Hackbusch's monograph \cite{Hackbusch1985}), and defining the norm \(\norm{\cdot}^2_{M,A}\) as
\begin{equation*}
	\norm{B}^2_{M,A} \coloneqq \sup_{\norm{v}_A \leq 1} \norm{Bv}^2_M, \quad \norm{w}^2_Q \coloneqq \langle w,\, w \rangle_Q \coloneqq \langle Q w,\, w \rangle,
\end{equation*}
we can formulate the Smoothing Property for a smoother \(S = p_k(MA)\) with respect to this norm and bound it as
\begin{equation}\label{eq:chebyshev_smoothing_bound}
	\norm{A S}^2_{M,A} = \sup_{\norm{v}_A \leq 1} \norm{A S v}^2_{M} = \sup_{\norm{v}_A \leq 1} \norm{A p_k(MA) v}^2_{M} \leq \max_{\lambda \in \sigma(MA)} \lambda p_k(\lambda)^2,
\end{equation}
where the last step follows by writing the test vectors \(v\) in an \(\langle \cdot,\cdot \rangle_A\)-orthonormal basis of eigenvectors of \(MA\) (and thus also of \(p_k(MA)\)). This points us towards the minimax problem
\begin{equation}\label{eq:chebyshev_smoother_minimization}
	\min_{\substack{p_k \in \mathbb{P}_k\\ p_k(0) = 1}} \max_{\lambda\in [\alpha,\, \beta]} \lambda^{1/2} |p_k(\lambda)|, \quad \beta > \alpha > 0,
\end{equation}
for an interval \([\alpha,\, \beta] \supset \sigma(MA)\). The following lemma gives us the solution.

\begin{lemma}\label{thm:weighted_minimax}
	Let \(k \geq 1\) be a positive integer. There exists \(\delta^{\star} \in (-1, 1)\) such that, for any \(\delta \in (-1, \delta^{\star})\), the minimax problem
	\begin{equation}\label{eq:weighted_minimax_lemma}
		\min_{\substack{p_k \in \mathbb{P}_k\\ p_k(-1) = 1}} \max_{t\in [\delta,\, 1]} (1 + t)^{1/2} |p_k(t)|,
	\end{equation}
	has a unique minimizer, which is given by
	\begin{equation*}
		\hat{p}_k(t) = \frac{1}{2k+1}\, W_k\left( -t \right),
	\end{equation*}
	where \(W_k\) is the \(k\)th-degree Chebyshev polynomial \emph{of the fourth kind}. Furthermore, this also holds for the limit case \(\delta = -1\), i.e.\ \(\hat{p}_k\) is also a solution to
	\begin{equation*}
		\min_{\substack{p_k \in \mathbb{P}_k\\ p_k(-1) = 1}} \max_{t\in [-1,\, 1]} (1 + t)^{1/2} |p_k(t)|.
	\end{equation*}
\end{lemma}

\begin{proof}
	Let us first note that, for \(V_k\) the \(k\)th-degree Chebyshev polynomial of the \emph{third} kind (see \cite{Mason2002, Mason2005}), the weighted polynomial \((1 + t)^{1/2} V_k(t)\) has \(k + 1\) alternation points \(t_i = \cos \theta_i\), \(i = 0, \dots, k\) with
	\begin{equation*}
		\theta_i = \frac{2i \pi}{2i + 1} \in [0, \pi],
	\end{equation*}
	and in particular \(t_i \in [\delta^{\star}, \, 1]\) with
	\begin{equation*}
		\delta^{\star} \coloneqq \min_{i \in \{0, \dots, k\}} t_i = -\cos \frac{\pi}{2k + 1} > -1.
	\end{equation*}
	Additionally, the Chebyshev polynomials of the third kind can be written in terms of those of the fourth kind as \(V_k(t) = (-1)^k W_k(-t)\).
	
	Now, for any \(\delta \in (-1, \delta^{\star})\), the set of functions \(\{(1+t)^{3/2} t^i\}_{i=0}^{k-1}\) is a Chebyshev system in \(C([\delta,1])\) (because the fixed root at \(-1\) lies outside the interval), and thus it follows from the Alternation Theorem that any minimizer \(\hat{q}_{k-1}\) of
	\begin{equation*}
		\min_{q_{k-1} \in \mathbb{P}_{k-1}} \max_{t\in [\delta,\, 1]} (1+t)^{1/2} |1 - (1+t) q_{k-1}(t)|,
	\end{equation*}
	equivalent to \eqref{eq:weighted_minimax_lemma} by rewriting \(p_k(t) = 1 + (1 + t) q_{k-1}(t)\), is characterized by
	\begin{equation*}
		(1+t)^{1/2} \left[1 - (1+t) \hat{q}_{k-1}(t)\right]
	\end{equation*}
	having \(k + 1\) alternation points within \([\delta, 1]\). This is in turn equivalent to \(1 - (1+t) \hat{q}_{k-1}(t) \propto V_k(t)\), which means that \(\hat{p}_k\) defined by
	\begin{equation*}
		\hat{p}_k(t) = \frac{V_k(t)}{V_k(-1)} = \frac{W_k(-t)}{W_k(1)}
	\end{equation*}
	is the unique minimizer of \eqref{eq:weighted_minimax_lemma}. The first claim then follows by noting that \(W_k(1) = 1 / (2k+1)\).
	
	Finally, since \((t + 1)^{1/2} |p_k(t)|\) is 0 at \(t=-1\) for any \(p_k \in \mathbb{P}_k\), \(\hat{p}_k\) is also the unique minimizer for the limit case \(\delta = -1\): indeed, any other candidate \(\tilde{p}_k\) would have to attain a maximum of \((t + 1)^{1/2} |\tilde{p}_k(t)|\) precisely at \(t=-1\), which implies \(\tilde{p}_k \equiv 0\), but then it clearly would not fulfill the \(\tilde{p}_k(-1) = 1\) constraint.
\end{proof}

\begin{remark}
	Note that this limit argument was not possible for the unweighted minimax problem and Chebyshev polynomials of the first kind, since the endpoints of the baseline interval \([-1,1]\) are alternation points of every \(T_k\).
\end{remark}

After a trivial reparametrization to the \([0, \, \beta]\) interval, \cref{thm:weighted_minimax} allows us to characterize the solution to \eqref{eq:chebyshev_smoother_minimization} for \(\alpha\) sufficiently small, and even 0, thus removing the need for an estimate of the lower end of the fine spectrum of \(MA\). In other words, for \(\beta > 0\), we have that
\begin{equation}\label{eq:chebyshev_4th_kind}
	\hat{W}_k \coloneqq \argmin_{\substack{p_k \in \mathbb{P}_k\\ p_k(0) = 1}} \max_{\lambda\in [0,\, \beta]} \lambda^{1/2} |p_k(\lambda)|, \quad \hat{W}_k(\zeta) = \frac{1}{2k + 1} W_k\left(1 - \frac{2}{\beta} \,\zeta\right).
\end{equation}

Finally, based on the recurrence relationship
\begin{align*}
	W_0(t) &= 1,\\
	W_1(t) &= 2t + 1,\\
	W_{n+1}(t) &= 2t\, W_n(t) - W_{n-1}(t), \quad n \geq 1,
\end{align*}
the same as with Chebyshev polynomials of the first kind but with a different linear polynomial, we can replicate the process from \cite{Gutknecht2002} to arrive at a two-term iteration. For simplicity, let us ignore preconditioning in our derivation. First of all, we may write the generic form of three-term Krylov iterations for \(n \geq 0\) as
\begin{equation*}
	\begin{aligned}
		r_{n+1} &= \frac{1}{\gamma_n} (A r_n - \alpha_n r_n - \beta_{n-1} r_{n-1}),\\
		x_{n+1} &= \frac{-1}{\gamma_n} (r_n + \alpha_n x_n + \beta_{n-1} x_{n-1}),
	\end{aligned}
\end{equation*}
with the consistency constraint that \(\gamma_n = -(\alpha_n + \beta_{n-1})\), and understanding \(r_{-1} \equiv 0\), \(x_{-1} \equiv 0\), \(\beta_{-1} = 0\). For the polynomial iteration based on Chebyshev polynomials of the fourth kind \eqref{eq:chebyshev_4th_kind}, we can write, for \(n \geq 1\) and with \(c = \beta / 2\),
\begin{multline*}
	\hat{W}_{n+1}(t) = \frac{1}{2n+3} W_{n+1}\left( 1 - \frac{t}{c} \right) =\\
	= \frac{1}{2n+3} \left[ 2 \left( 1 - \frac{t}{c} \right) W_{n}\left( 1 - \frac{t}{c} \right) - W_{n-1}\left( 1 - \frac{t}{c} \right) \right] =\\
	= 2\, \frac{2n+1}{2n+3} \, \left( 1 - \frac{t}{c} \right)\, \hat{W}_{n}(t) - \frac{2n-1}{2n+3} \, \hat{W}_{n-1}(t),
\end{multline*}
which leads us to
\begin{equation*}
	r_{n+1} = 2\, \frac{2n+1}{2n+3}\, r_n - \frac{2}{c}\, \frac{2n+1}{2n+3} \, A r_n - \frac{2n-1}{2n+3}\, r_{n-1}, \quad n \geq 1,
\end{equation*}
and thus
\begin{equation*}
	\gamma_n = - \frac{2n+3}{2n+1} \, \frac{c}{2}, \quad \alpha_n = c, \quad \beta_{n-1} = - \frac{2n-1}{2n+1} \, \frac{c}{2}, \quad n\geq 1.
\end{equation*}
The remaining parameter values can be extracted from
\begin{equation*}
	r_{1} = \frac{1}{3}\,  W_{1}\left( I - \frac{1}{c} \, A \right) r_0 = r_0 - \frac{2}{3c} \, A \, r_0,
\end{equation*}
from which we can infer that \(\gamma_0 = -\frac{3c}{2} = - \alpha_0\). At this point, following \cite[Section~5]{Gutknecht2002}, we can derive a two-term recurrence based on the parameters
\begin{equation*}
	\psi_n = \frac{\beta_n}{\gamma_n} = - \left(\frac{2n+1}{2n+3}\right)^2, \quad \omega_n = - \frac{1}{\gamma_n} = \frac{2n+1}{2n+3}\, \frac{2}{c}, \quad n\geq 0,
\end{equation*}
as
\begin{equation*}
	\left\{
	\begin{aligned}
		v_n &= r_n - \psi_{n-1} v_{n-1},\\
		x_{n+1} &= x_n + \omega_n v_{n},\\
		r_{n+1} &= b - A x_{n+1}.
	\end{aligned}
	\right.
\end{equation*}
for \(n \geq 0\), understanding \(\psi_{-1} = 0\) and \(v_{n-1}\equiv 0\). Denoting \(\mu_n = \frac{2n+1}{2n+3}\) for \(n \geq 0\) and realizing that \(\mu_{n+1} = (2 - \mu_{n})^{-1}\), we can finally write the (preconditioned) polynomial iteration based on Chebyshev polynomials of the fourth kind as in \cref{alg:chebyshev_iteration_4th}.

\begin{algorithm}[!htb]
	\caption{Chebyshev iteration of the fourth kind}\label{alg:chebyshev_iteration_4th}
	\begin{algorithmic}[1]
		\Require{A linear system \(A x = b\) with a s.p.d. matrix \(A\), and an equally s.p.d. preconditioner \(M\). An initial tentative solution \(x_0\). An upper bound on the spectrum of \(MA\), \(\beta \geq \lambda_{max}(MA)\). A positive integer \(k \geq 1\).}

        \vspace{1em}

		\State{Initialize \(r = b - Ax_0\), \(p = Mr\).}
		\State{Initialize \(\mu = 1/3\).}
		\State{Set \(d = 4 \beta^{-1}\).}

        \vspace{1em}

		\State{Initialize \(x = x_0 + d \mu \, p\)}

        \vspace{1em}

		\For{\(i=2,\dots,k\)}
		
		\State{\(r \gets b - Ax\)}
		\State{\(p \gets \mu^2 \, p + Mr\)}
		\State{\(\mu \gets \left(2 - \mu \right)^{-1}\)}

		\State{\(x \gets x + d \mu \, p\)}
		
		\EndFor

        \vspace{1em}

		\Return{\(x\)}
	\end{algorithmic}
\end{algorithm}

\section{The Partition of Unity Method}\label{sec:pum}

For the experiments of this section, we will rely on the Partition of Unity Method of Schweitzer and Griebel \cite{Schweitzer2003, Griebel2002}. Let us succinctly introduce it. The method relies on sets of functions \(\{\varphi_i\}_{i=1}^n\), \(\varphi_i \colon \bar{\Omega} \to \mathbb{R}\), forming a \emph{partition of unity} (PU) over some Lipschitz domain \(\Omega\subset\mathbb{R}^d\), meaning that
\begin{equation*}
	0 \leq \varphi_i(x) \leq 1 \quad \forall i \in \{1,\dots,n\}, \quad \text{and} \quad \sum_{i=1}^n \varphi_i(x) = 1,\quad \forall x \in \bar{\Omega},
\end{equation*}
which are used to generate discrete function spaces as
\begin{equation*}
	V^{\mathrm{PU}}(\Omega) = \sum_{i=1}^n \varphi_i V_i(\omega_i),\quad \omega_i \coloneqq \supp(\varphi_i),
\end{equation*}
where each \(V_i(\omega_i)\), called \emph{local space}, is defined on the corresponding \(\omega_i\), all of them being ``glued together'' by the PU functions. Usually \(\Omega\) and \(\omega_i\) are dropped from the notation. The simplest choice for the local spaces \(V_i\) is to make them polynomial spaces of degree \(p_i\) (although more complicated spaces are naturally allowed), while the PU functions \(\varphi_i\) can be constructed to satisfy \(\varphi_i \in C^{k}(\Omega)\) for any predefined \(k \in \mathbb{N}\) (or be simply piecewise constant) \cite{JimenezRecio2024}. If we denote \(\{\vartheta_i^{k}\}_{k=1}^{m_i}\) a basis for each local space \(V_i\),
\begin{equation*}
	\left\{\varphi_i \vartheta_i^{k} \,\colon\, i = 1,\dots,n; \, k = 1, \dots, m_i\right\}
\end{equation*}
is a global basis of \(V^{\mathrm{PU}}\), and thus any classical Galerkin problem of the type
\begin{equation*}
	u\in V^{\mathrm{PU}} \quad\colon\quad a(v, u) = \ell(v), \quad \forall v\in V^{\mathrm{PU}},
\end{equation*}
can be represented by a linear system \(A \bar{u} = b\), where \(u = \sum_{i=1}^n \sum_{k=1}^{m_i} \bar{u}_{(i,k)} \varphi_i  \vartheta^{k}_{i}\) and the stiffness matrix \(A\) and load vector \(b\) are given by
\begin{equation*}
	A_{(i,k),(j,l)} = a(\varphi_i \vartheta_i^{k}, \,\varphi_j \vartheta_j^{l}),\quad b_{(i,k)} = \ell(\varphi_i \vartheta_i^{k}).
\end{equation*}
In particular, \(A\) fits \cref{def:block_structure} with the blocks \(A_{ij} = (A_{(i,k),(j,l)}) \in \mathbb{R}^{m_i \times m_j}\). The sparsity of \(A\) is achieved through a particular construction of the PU \(\{\varphi_i\}\), which in practice is constructed based on an \emph{open cover} of the domain \(\Omega\),
\begin{equation*}
	\mathcal{C}(\Omega) = \{\omega_i\}_{i=1}^n \quad\colon\quad \Omega \subset \bigcup_{i=1}^n \omega_i, \quad \supp(\varphi_i) = \omega_i,
\end{equation*}
whose elements we call \emph{patches}, which is based on a regular grid for a bounding box of \(\Omega\). In \cref{fig:pum_cover} we show such a grid and the resulting cover for \(\Omega = (0, 1)^2\) and uniform refinement level 3 (which corresponds to \(2^3\cdot 2^3 = 64\) patches). A progressive refinement of such regular grids allows us to construct a \emph{nonnested} sequence of function spaces
\begin{equation*}
	V_0^{\mathrm{PU}} \not\subset V_1^{\mathrm{PU}} \not\subset \cdots \not\subset V_{J-1}^{\mathrm{PU}} \not\subset V_{J}^{\mathrm{PU}},
\end{equation*}
where we can write each space as
\begin{equation*}
	V_{l}^{\mathrm{PU}} = \sum_{i=1}^{n_{l}} \varphi^{(l)}_{i} V^{(l)}_{i}\left( \omega^{(l)}_{i} \right), \quad \omega^{(l)}_{i} \coloneqq \supp \left( \varphi^{(l)}_{i} \right), \quad l = 0,\dots,J.
\end{equation*}
This nonnestedness makes the prolongation operators \(I_{l-1}^{l} \,\colon\, V_{l-1}^{\mathrm{PU}} \to V_{l}^{\mathrm{PU}}\) and the restriction operators \(I^{l-1}_{l} \,\colon\, V_{l}^{\mathrm{PU}} \to V_{l-1}^{\mathrm{PU}}\) nontrivial to define. Fortunately, they are also not too difficult: we can define the prolongation operators via a global-to-local \(L^2\)-projection (see \cite{Schweitzer2003}), i.e.
\begin{align*}
	I_{l-1}^{l} \,\colon\, V_{l-1}^{\mathrm{PU}} &\to V_{l}^{\mathrm{PU}}\\
	u^{(l-1)} &\mapsto I_{l-1}^{l} u^{(l-1)} = \sum_{i=1}^{n_{l}} \varphi^{(l)}_{i} u^{(l)}_i
\end{align*}
where each \(u^{(l)}_i \in V^{(l)}_{i}\), \(i=1,\dots,n_{l}\), solves
\begin{equation*}
	\langle u^{(l)}_i, v \rangle_{L^2\left(\omega^{(l)}_{i}\right)} = \langle u^{(l-1)}, v \rangle_{L^2\left(\omega^{(l)}_{i}\right)}, \quad \forall v \in V^{(l)}_{i},
\end{equation*}
and the restriction operators can then be defined by transposition. For more details, we refer once again to \cite{Schweitzer2003, Griebel2002, JimenezRecio2024}.

\begin{figure}[tbhp]
	\centering
	\begin{tikzpicture}[scale=0.8]
		\foreach \i in {0,...,8} {
			\draw [very thin,black] (\i * 0.5,0) -- (\i * 0.5,4);
		}
		\foreach \i in {0,...,8} {
			\draw [very thin,black] (0,\i * 0.5) -- (4,\i * 0.5);
		}

		\foreach \i in {0,8} {
			\draw [very thin,black] (6 + \i * 0.5,0) -- (6 + \i * 0.5,4);
		}
		\foreach \i in {0,8} {
			\draw [very thin,black] (6 + 0,\i * 0.5) -- (6 + 4,\i * 0.5);
		}
		
		\foreach \i in {1,...,7} {
			\def\xmin{{6 + 0.5 * (\i - 0.15)}}
			\def\xmax{{6 + 0.5 * (\i + 0.15)}}
			
			\draw [ultra thin,dashed] (\xmin,0) -- (\xmin,4);
			\draw [ultra thin,dashed] (\xmax,0) -- (\xmax,4);

			\fill[black, opacity=0.15] (\xmax,0) rectangle (\xmin,4);
		}
		\foreach \i in {1,...,7} {
			\def\ymin{{(\i - 0.15) * 0.5}}
			\def\ymax{{(\i + 0.15) * 0.5}}
			
			\draw [ultra thin,dashed] (6,\ymin) -- (10,\ymin);
			\draw [ultra thin,dashed] (6,\ymax) -- (10,\ymax);

			\fill[black, opacity=0.15] (6,\ymin) rectangle (10,\ymax);
		}
		
		\foreach \i in {0,...,7} {
			\def\xpos{{6 + 0.5 * (\i + 0.5)}}
			\foreach \j in {0,...,7} {
				\def\ypos{{0.5 * (\j + 0.5)}}
				\node[draw,circle,fill=black,inner sep=0pt,minimum width=1pt] at (\xpos,\ypos) {};
			}
		}
	\end{tikzpicture}

	\caption{On the left, a regular grid for \(\Omega = (0,1)^2\) at discretization level 3, and on the right, the associated open cover generated by stretching the square elements of the grid. Each black dot corresponds to a patch center, patch support boundaries are dashed, and the regions where multiple patches overlap are colored in gray.}
	\label{fig:pum_cover}
\end{figure}
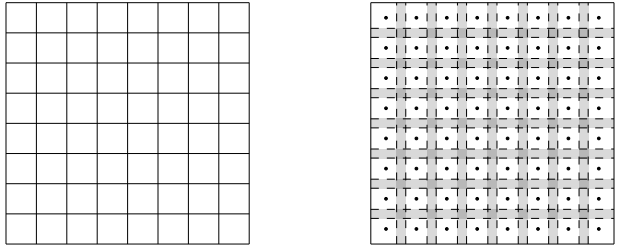

\subsection{Block-FSAI in the PUM framework}

Let us finally mention the similarity between the block-FSAI preconditioner for the PUM and the \emph{multilevel overlapping Schwarz} (MOS) smoother introduced for the PUM in \cite{Griebel2005}. Indeed, for a system \(A \bar{u} = b\), we can understand the action of the FSAI preconditioner \(F^{\top} S^{-1} F\) in two stages, the first one being \(\bar{z} = S^{-1} F b\), and the second one \(\bar{u} = F^{\top} \bar{z}\). The first step is equivalent to solving
\begin{equation*}
	A[\mathcal{P}_k, \mathcal{P}_k] \bar{w} = b[\mathcal{P}_k]
\end{equation*}
for every block-row index \(k\) and keeping only the entry from \(\bar{w}\) corresponding to index \(k\) in each case, which becomes \(\bar{z}_k\). In terms of the weak form and a PUM space \(V^{\mathrm{PU}} = \sum_{i=1}^n \varphi_i V_i\), this corresponds to solving the PDE in each \emph{enlarged local space} \(\hat{V}_k \coloneqq \sum_{i \in \mathcal{P}_k} \varphi_i V_i\), i.e.
\begin{equation*}
	w = \varphi_k w_k + \sum_{i \in \tilde{\mathcal{P}}_k} \varphi_i \tilde{w}_i \in \hat{V}_k \quad \colon \quad a(w, v) = \ell(v), \quad \forall v\in \hat{V}_k,
\end{equation*}
from which we keep only \(w_k\) (whose coefficients in the basis of \(V_k\) constitute \(\bar{z}_k\)). At the second step, every \(\bar{z}_k\) contributes directly to the final corresponding entry \(\bar{u}_k\), but also to the entries \(\bar{u}[\tilde{\mathcal{P}}_k]\) via the correction \(- A[\tilde{\mathcal{P}}_k, \tilde{\mathcal{P}}_k]^{-1} A[\tilde{\mathcal{P}}_k, \{k\}] \bar{z}_k\). This is equivalent to setting \(u = \sum_{k=1}^n q_k \in V^{\mathrm{PU}}\) where each \(q_k \in \hat{V}_k\) is of the form
\begin{equation*}
	q_k = \varphi_k w_k + \tilde{q}_k, \quad \tilde{q}_k \in \tilde{V}_k \coloneqq \sum_{i \in \tilde{\mathcal{P}}_k} \varphi_i V_i
\end{equation*}
(note that \(\hat{V}_k = \varphi_k V_k + \tilde{V}_k\)) and solves
\begin{equation*}
	a(q_k, v) = 0, \quad \forall v \in \tilde{V}_k.
\end{equation*}
Each \(\tilde{q}_k\) can thus be understood as a correction to \(\sum_{i=1}^n \varphi_i w_i\), which accounts for the fact that, for every \(k^{\prime} \in \tilde{\mathcal{P}}_k\), the solution step for \(w_{k^{\prime}}\) had not considered the local space \(V_k\) (because \(k > k^{\prime}\)).

While the FSAI preconditioner allows for an adaptive construction of the enlarged local spaces, the enlarged local spaces of the MOS smoother
\begin{equation*}
	W_i \coloneqq \sum_{j\in \mathcal{N}_i} \varphi_j V_j, \quad \mathcal{N}_i \coloneqq \{j \,\colon\, \omega_i \cap \omega_j \neq \emptyset\},
\end{equation*}
are based on a geometric criterion and their size is fixed a priori. This makes it prohibitively expensive in certain scenarios, particularly in 3-dimensional problems (note that \(\operatorname{card}(\mathcal{N}_i) = 3^d\) for homogeneously refined covers in \(\mathbb{R}^d\)). FSAI adaptivity, on the contrary, allows us to control the size of the enlarged local spaces, while also constructing them in an algebraic manner, based solely on information from the matrix itself.

\section{Numerical experiments}\label{sec:experiments}

Let us first describe the PUM spaces that we will use in our tests. We will use PU functions \(\{\varphi_i\}\) satisfying the regularity requirements of each problem, with the local spaces \(V_i\) being formed by local polynomials of degree up to \(p\), i.e.\ for \(\Omega \subset \mathbb{R}^2\)
\begin{equation*}
	V_i = \operatorname{span} \langle x^a y^b \,\colon\, a, b \geq 0,\, a + b \leq p \rangle,
\end{equation*}
and for \(\Omega \subset \mathbb{R}^3\)
\begin{equation*}
	V_i = \operatorname{span} \langle x^a y^b z^c \,\colon\, a,b,c \geq 0,\, a + b + c \leq p \rangle,
\end{equation*}
so that in particular \(\dim(V_i) = \binom{p+d}{d}\) with \(\Omega \subset \mathbb{R}^d\). In some cases, the polynomial degree of local spaces \(V_i\) for which \(\supp(\varphi_i) \cap \partial\Omega \neq \emptyset\) will be increased to accomodate for higher derivatives appearing in the boundary integrals. For the multilevel iteration, we will construct the coarse operators via Galerkin products, i.e.\ \(A_{l-1} = R_l A_l P_l\) with \(P_l\) being the prolongation matrix from level \(l-1\) to level \(l\), and \(R_l = P_l^{\top}\).

With respect to the block-FSAI preconditioners, we will restrict ourselves to adaptively-constructed sparsity patterns, either nonnested or with 1 level of nesting. In the nested case, the adaptivity parameters will be the same for the two levels. For the adaptive algorithm, we choose the trivial tolerance parameter \(\tau = 1\), allowing ourselves to modify the number of adaptive steps \(t_{max}\). Additionally, as initial pattern \(\mathcal{P}_0\) for \cref{alg:adaptive_fsai} we will always choose the diagonal pattern.

For the Chebyshev iteration, we will estimate the maximum eigenvalue of the FSAI-preconditioned matrices with the Lanczos algorithm \cite{vanderVorst1982}, for an iteration number (and size of the resulting tridiagonal matrix) of 100, and we will multiply the obtained estimate by a factor of \(1.01\). Furthermore, in addition to the multilevel \(V(k, k)\)-cycle, following the work of Phillips and Fischer \cite{Phillips2022} we will also evaluate the performance of the non-symmetric \(V(2k, 0)\)-cycle, justified by the optimality of the polynomial smoother.

\begin{example}\label{ex:biharmonic_square}

	To begin with, we consider the biharmonic equation on the unit square \(\Omega = (0,1)^2\), with essential boundary conditions on the whole boundary
	\begin{equation*}
		\left\{
		\begin{aligned}
			\Delta^2 u &= f \quad &&\text{ in } \Omega,\\
			u &= g_0 \quad &&\text{ on } \partial\Omega,\\
			\nabla u \cdot \mathbf{n} &= g_1 \quad &&\text{ on } \partial\Omega.
		\end{aligned}
		\right.
	\end{equation*}
	For a discrete space \(V_n\subset H^2(\Omega)\), and following Nitsche's method for the weak imposition of boundary conditions, the weak formulation consists in finding \(u\in V_n\) such that
	\begin{equation*}
		a_n(v, u) = \ell_n(v),\quad \forall v \in V_n,
	\end{equation*}
	where
	\begin{equation}\label{eq:biharmonic_nitsche}
	\begin{aligned}
		a_n(v, u) =& \int_{\Omega} \Delta u \Delta v \diff\Omega + \int_{\partial\Omega} \left(\gamma_n^{(0)} u v + v \nabla(\Delta u) \cdot \mathbf{n} + u \nabla(\Delta v) \cdot \mathbf{n} \right) \diff\Gamma +\\
		&+ \int_{\partial\Omega} \left( \gamma_n^{(1)} (\nabla u \cdot \mathbf{n}) (\nabla v \cdot \mathbf{n}) - \Delta u (\nabla v \cdot \mathbf{n}) - \Delta v (\nabla u \cdot \mathbf{n}) \right) \diff\Gamma,\\
		\ell_n(v) =& \int_{\Omega} f v \diff\Omega + \int_{\partial\Omega} g_0 \left(\gamma_n^{(0)} v + \nabla(\Delta v) \cdot \mathbf{n} \right) \diff\Gamma +\\
		&+ \int_{\partial\Omega} g_1 \left( \gamma_n^{(1)} (\nabla v \cdot \mathbf{n}) - \Delta v \right) \diff\Gamma.
	\end{aligned}
	\end{equation}
	The stabilization functions \(\gamma_n^{(0)}\), \(\gamma_n^{(1)}\) depend on \(V_n\) (hence our \(n\) subscript) and have to be ``large enough'' to ensure that the bilinear form is elliptic. We refer to \cite{JimenezRecio2024} for the actual criteria and how we enforce them in the PUM.
	
	For this simple domain and with homogeneous refinement, each PUM space \(V^{\mathrm{PU}}_l\) consists of \(2^l \cdot 2^l = 2^{2l}\) PU functions and their corresponding local spaces (recall \cref{fig:pum_cover} illustrating the case of level \(l=3\)). Since \(\dim(V_i) = \binom{p_i+2}{2}\), this means that \(\dim(V^{\mathrm{PU}}_l) = \binom{p+2}{2} 2^{2l}\) when using the same polynomial degree \(p\) for all local spaces.
	
	For this initial test, we construct a multilevel sequence from coarsest level 2 up to finest level 7, and local polynomial spaces of degree 2. For patches intersecting the boundary, we increase the degree to 3 to accomodate for third-order derivatives in the boundary integrals. At the finest level, this yields \(100,336\) degrees of freedom. To measure the multilevel convergence rates, we consider the reference manufactured solution
	\begin{equation*}
		\hat{u}(x, y) = \cos(2 \pi x) \sin(2 \pi y),
	\end{equation*}
	and obtain the corresponding discrete solution \(u_{\mathrm{PU}} \in V^{\mathrm{PU}}\) with a direct solver. In order to measure the convergence rates of the multilevel iteration, we choose an initial iterate \(u^{(0)} = u_{\mathrm{PU}} + e^{(0)}\), where the coefficient vector of \(e^{(0)} \in V^{\mathrm{PU}}\) is generated with random entries in \((-1, 1)\), and normalized so that \(\norm{e^{(0)}}_{L^2(\Omega)} = 1\). We then measure convergence rates for the \(L^2(\Omega)\) and the energy norms, which we can do with the corresponding error coefficient vectors \(\bar{e}^{(n)} = \bar{u}^{(n)} - \bar{u}_{\mathrm{PU}}\), the mass matrix \(M\) and the stiffness matrix \(A\) as
	\begin{equation*}
		\rho_{L^2} = \left(\frac{\langle M \bar{e}^{(n)}, \bar{e}^{(n)} \rangle}{\langle M \bar{e}^{(0)}, \bar{e}^{(0)} \rangle} \right)^{1/(2n)} = \langle M \bar{e}^{(n)}, \bar{e}^{(n)} \rangle^{1/(2n)}, \quad \rho_{A} = \left(\frac{\langle A \bar{e}^{(n)}, \bar{e}^{(n)} \rangle}{\langle A \bar{e}^{(0)}, \bar{e}^{(0)} \rangle} \right)^{1/(2n)},
	\end{equation*}
	for a number of iterations \(n\). We iterate until \(\norm{e^{(n)}}_{L^2(\Omega)} = \langle M \bar{e}^{(n)}, \bar{e}^{(n)} \rangle^{1/2} < 10^{-8}\) or a maximum number of 50 iterations is reached. In our first case, we also look at convergence rates of the residual \(r^{(n)} = b - A \bar{u}^{(n)}\) in the \(l^2\)-vector norm, which is equivalent to the \(A^2\) norm of the error \(\bar{e}^{n}\).
	
	In \cref{fig:convergence_rates_square} we provide the measured convergence rates for smoothing steps \(k \in \{1,\dots, 10\}\) and FSAI adaptive steps \(t_{max}\in \{4,\dots,9\}\), with and without a further nested FSAI preconditioner. As we can see, all convergence rates depend similarly on \(t_{max}\) and \(k\), and thus from now on (unless explicitly stated otherwise) we will only provide convergence rates in the energy norm, \(\rho_A\).
	
	\begin{figure}[tbhp]
		\centering
		\begin{tikzpicture}
			\begin{groupplot}[group style={group size= 2 by 3, vertical sep=0.75cm, horizontal sep=0.5cm},height=4.25cm,width=6.5cm]
				\nextgroupplot[grid=major,
					title={Adaptive FSAI},
					ylabel={\(\rho_{L^2}\)},
					ymode=log,
					ytick={0.12, 0.25, 0.50, 1.00}, yticklabels={0.12, 0.25, 0.50, 1.00}, ymin=0.075, ymax=1.1,
					legend entries={4,5,6,7,8,9}, legend pos=south west,
					legend style={name=legend, draw=none}, legend columns=2];

					\pgfplotsforeachungrouped \n in {4,5,...,9}{
						\addplot table [x=k, y=a\n_r0_L2_norm, col sep=comma] {bh_sq_pbdry_with_residual.csv};
					};

				\nextgroupplot[grid=major,
					title={Nested adaptive FSAI},
					ymode=log,
					ytick={0.12, 0.25, 0.50, 1.00}, yticklabels={}, ymin=0.075, ymax=1.1];
					
					\pgfplotsforeachungrouped \n in {4,5,...,9}{
						\addplot table [x=k, y=a\n_r1_L2_norm, col sep=comma] {bh_sq_pbdry_with_residual.csv};
					};
					
				\nextgroupplot[grid=major,
					ylabel={\(\rho_{A}\)},
					ymode=log,
					ytick={0.05, 0.12, 0.25, 0.50, 1.00}, yticklabels={0.05, 0.12, 0.25, 0.50, 1.00}, ymin=0.025, ymax=1.1];
					
				\pgfplotsforeachungrouped \n in {4,5,...,9}{
					\addplot table [x=k, y=a\n_r0_energy_norm, col sep=comma] {bh_sq_pbdry_with_residual.csv};
				};
				
				\nextgroupplot[grid=major,
					ymode=log,
					ytick={0.05, 0.12, 0.25, 0.50, 1.00}, yticklabels={}, ymin=0.025, ymax=1.1];
				
				\pgfplotsforeachungrouped \n in {4,5,...,9}{
					\addplot table [x=k, y=a\n_r1_energy_norm, col sep=comma] {bh_sq_pbdry_with_residual.csv};
				};
				
				\nextgroupplot[grid=major,
				xlabel={\(k\)}, ylabel={\(\rho_{r}\)},
				ymode=log,
				ytick={0.02, 0.05, 0.12, 0.25, 0.50, 1.00}, yticklabels={0.02, 0.05, 0.12, 0.25, 0.50, 1.00}, ymin=0.015, ymax=1.1];
				
				\pgfplotsforeachungrouped \n in {4,5,...,9}{
					\addplot table [x=k, y=a\n_r0_residual, col sep=comma] {bh_sq_pbdry_with_residual.csv};
				};
				
				\nextgroupplot[grid=major,
				xlabel={\(k\)},
				ymode=log,
				ytick={0.02, 0.05, 0.12, 0.25, 0.50, 1.00}, yticklabels={}, ymin=0.015, ymax=1.1];
				
				\pgfplotsforeachungrouped \n in {4,5,...,9}{
					\addplot table [x=k, y=a\n_r1_residual, col sep=comma] {bh_sq_pbdry_with_residual.csv};
				};
			\end{groupplot}
			
			\node [above,text width=1.82cm,fill=white,align=center] (legendtitle) at (legend.north) {\(t_{max}\)};
			\node [fit=(legendtitle)(legend),draw,inner sep=0pt] {};
		\end{tikzpicture}
		\caption{\(V(k,k)\)-cycle convergence rates in the \(L^2(\Omega)\) and energy norms, as well as the \(l^2\) norm of the residual, for \cref{ex:biharmonic_square}, with local polynomial spaces of degree \(p = 2\) (refined to \(p = 3\) for patches intersecting the boundary).}
		\label{fig:convergence_rates_square}
	\end{figure}
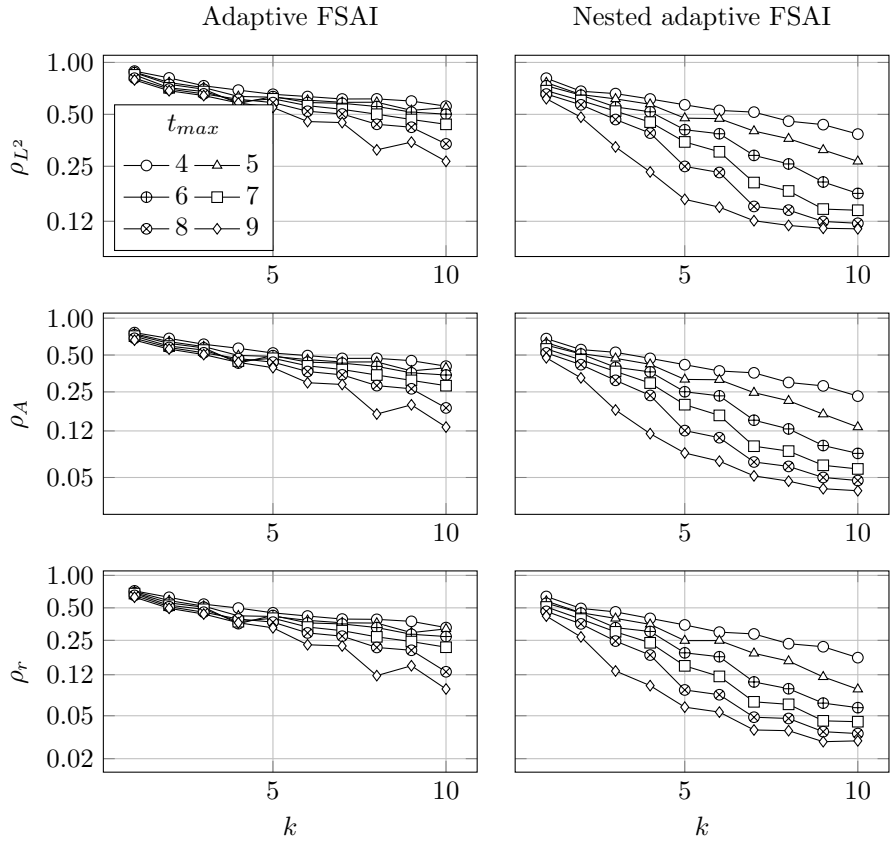

	We rely on nonnested FSAI smoothers for comparison of the \(V(k, k)\) and the \(V(2k,0)\) cycles, whose results we present in \cref{fig:sym_vs_nonsym_chebyshev}. As expected, each nonsymmetric cycle outperforms the corresponding symmetric counterpart.

	\begin{figure}[tbhp]
		\centering
		\begin{tikzpicture}
			\begin{groupplot}[group style={group size= 3 by 2, horizontal sep=0.5cm},height=3.75cm,width=5cm]
				
				\nextgroupplot[grid=major,
				title={\(t_{max}=4\)},
				ylabel={\(\rho_{A}\)},
				ymode=log,
				ytick={0.06, 0.12, 0.25, 0.50, 1.00}, yticklabels={0.06, 0.12, 0.25, 0.50, 1.00}, ymin=0.05, ymax=1.1,
				legend entries={\(V(k,k)\)\\\(V(2k,0)\)\\}, legend pos=south west];

				\addplot table [x=k, y=a4_r0_energy_norm, col sep=comma] {bh_sq_pbdry_with_residual.csv};
				
				\addplot table [x=k, y=a4_r0_energy_norm, col sep=comma] {bh_sq_nonsym.csv};
				
				\nextgroupplot[grid=major,
				title={\(t_{max}=5\)},
				ymode=log,
				ytick={0.06, 0.12, 0.25, 0.50, 1.00}, yticklabels={}, ymin=0.05, ymax=1.1];
				
				\addplot table [x=k, y=a5_r0_energy_norm, col sep=comma] {bh_sq_pbdry_with_residual.csv};
				
				\addplot table [x=k, y=a5_r0_energy_norm, col sep=comma] {bh_sq_nonsym.csv};
				
				\nextgroupplot[grid=major,
				title={\(t_{max}=6\)},
				ymode=log,
				ytick={0.06, 0.12, 0.25, 0.50, 1.00}, yticklabels={}, ymin=0.05, ymax=1.1];
				
				\addplot table [x=k, y=a6_r0_energy_norm, col sep=comma] {bh_sq_pbdry_with_residual.csv};
				
				\addplot table [x=k, y=a6_r0_energy_norm, col sep=comma] {bh_sq_nonsym.csv};
				
				\nextgroupplot[grid=major,
				title={\(t_{max}=7\)},
				xlabel={\(k\)}, ylabel={\(\rho_{A}\)},
				ymode=log,
				ytick={0.06, 0.12, 0.25, 0.50, 1.00}, yticklabels={0.06, 0.12, 0.25, 0.50, 1.00}, ymin=0.05, ymax=1.1];
				
				\addplot table [x=k, y=a7_r0_energy_norm, col sep=comma] {bh_sq_pbdry_with_residual.csv};
				
				\addplot table [x=k, y=a7_r0_energy_norm, col sep=comma] {bh_sq_nonsym.csv};
				
				\nextgroupplot[grid=major,
				title={\(t_{max}=8\)},
				xlabel={\(k\)},
				ymode=log,
				ytick={0.06, 0.12, 0.25, 0.50, 1.00}, yticklabels={}, ymin=0.05, ymax=1.1];
				
				\addplot table [x=k, y=a8_r0_energy_norm, col sep=comma] {bh_sq_pbdry_with_residual.csv};
				
				\addplot table [x=k, y=a8_r0_energy_norm, col sep=comma] {bh_sq_nonsym.csv};
				
				\nextgroupplot[grid=major,
				title={\(t_{max}=9\)},
				xlabel={\(k\)},
				ymode=log,
				ytick={0.06, 0.12, 0.25, 0.50, 1.00}, yticklabels={}, ymin=0.05, ymax=1.1];
				
				\addplot table [x=k, y=a9_r0_energy_norm, col sep=comma] {bh_sq_pbdry_with_residual.csv};
				
				\addplot table [x=k, y=a9_r0_energy_norm, col sep=comma] {bh_sq_nonsym.csv};

			\end{groupplot}
		\end{tikzpicture}
		\caption{Comparison of the (Chebyshev-smoothing) \(V(k,k)\)-cycle and \(V(2k,0)\)-cycle convergence rates for nonnested adaptive FSAI smoothers, for \cref{ex:biharmonic_square}.}
		\label{fig:sym_vs_nonsym_chebyshev}
	\end{figure}
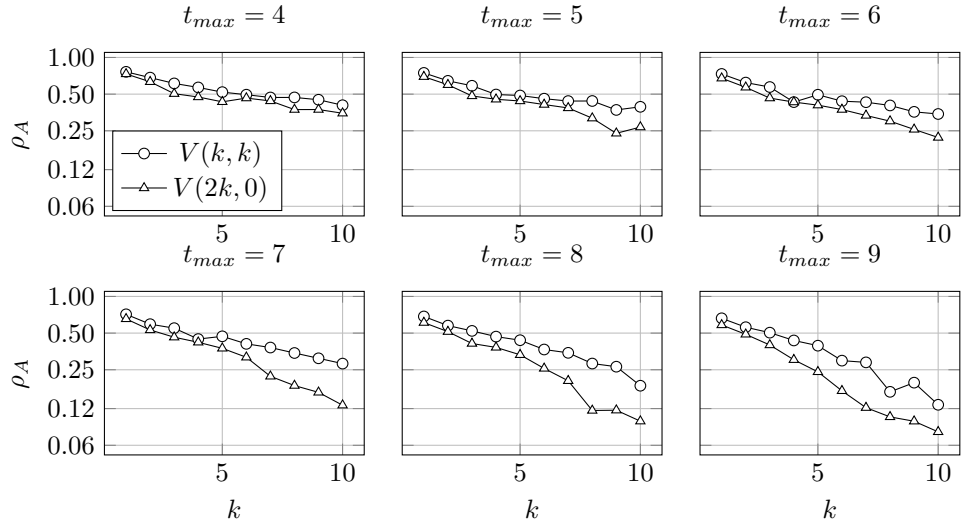
	
	Let us further mention that, in the case of nonnested FSAI, the usual Richardson smoothing cycles have only shown systematic proper convergence for \(t_{max} = 9\). The results for nested FSAI are presented in \cref{fig:chebyshev_vs_richardson}, revealing the superiority of Chebyshev smoothing also in this case.

	\begin{figure}[tbhp]
		\centering
		\begin{tikzpicture}
			\begin{groupplot}[group style={group size= 2 by 1, horizontal sep=0.5cm},height=5cm,width=6.25cm]
				\nextgroupplot[grid=major,
				title={Richardson smoothing},
				xlabel={\(k\)}, ylabel={\(\rho_{A}\)},
				ymode=log,
				ytick={0.05, 0.10, 0.25, 0.50, 1.00}, yticklabels={0.05, 0.10, 0.25, 0.50, 1.00}, ymin=0.025, ymax=1.1];
				
				\pgfplotsforeachungrouped \n in {4,5,...,9}{
					\addplot table [x=k, y=a\n_r1_energy_norm, col sep=comma] {bh_sq_richardson.csv};
				};
				
				\nextgroupplot[grid=major,
				title={Chebyshev smoothing},
				xlabel={\(k\)},
				ymode=log,
				ytick={0.06, 0.10, 0.25, 0.50, 1.00}, yticklabels={}, ymin=0.025, ymax=1.1,
				legend entries={4,5,6,7,8,9}, legend pos=outer north east,
				legend style={at={(1.05,0.935)},name=legend, draw=none}];
				
				\pgfplotsforeachungrouped \n in {4,5,...,9}{
					\addplot table [x=k, y=a\n_r1_energy_norm, col sep=comma] {bh_sq_pbdry_with_residual.csv};
				};
				
			\end{groupplot}
			
			\node [above,fill=white,align=center] (legendtitle) at (legend.north) {\(t_{max}\)};
			\node [fit=(legendtitle)(legend),draw,inner sep=0pt] {};
		\end{tikzpicture}
		\caption{Comparison of the \(V(k,k)\)-cycle convergence rates for nested adaptive FSAI smoothers performed either with Richardson smoothing or with Chebyshev smoothing, for \cref{ex:biharmonic_square}.}
		\label{fig:chebyshev_vs_richardson}
	\end{figure}
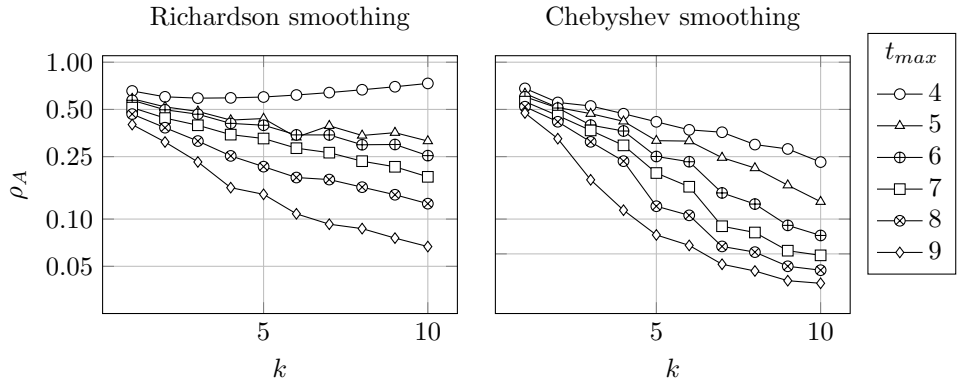
	
	Finally, we measure the convergence rates when choosing local spaces of polynomial degree \(p \in \{3, 4, 5\}\). It suffices to compare the results for nonnested FSAI, which we do in \cref{fig:A_convergence_rates_square_p} for \(k \in [1,5]\). As it can be clearly seen, the convergence rates do vary with the polynomial degree, but they decrease for increasing \(p\) (or remain similar for \(p\in\{4,5\}\) and \(t_{max} \geq 5\)).

	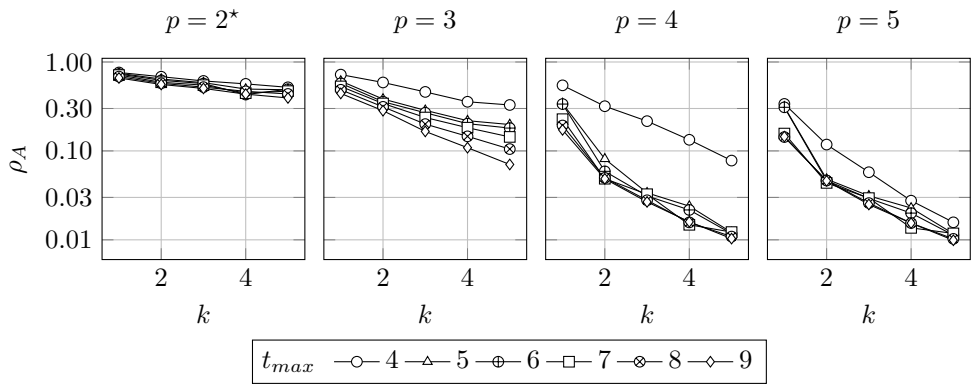
\begin{figure}[tbhp]
		\centering
		\begin{tikzpicture}
			\begin{groupplot}[group style={group size= 4 by 1, horizontal sep=0.25cm},height=4.25cm,width=4.25cm]
				\nextgroupplot[grid=major,
				title={\(p = 2^{\star}\)},
				xlabel={\(k\)}, ylabel={\(\rho_{A}\)},
				restrict x to domain=1:5,
				ymode=log,
				ytick={0.01, 0.03, 0.10, 0.30, 1.00}, yticklabels={0.01, 0.03, 0.10, 0.30, 1.00}, ymin=0.006, ymax=1.1];
				
				\pgfplotsforeachungrouped \n in {4,5,...,9}{
					\addplot table [x=k, y=a\n_r0_energy_norm, col sep=comma] {bh_sq_pbdry_with_residual.csv};
				};

				\nextgroupplot[grid=major,
				title={\(p = 3\)},
				xlabel={\(k\)},
				restrict x to domain=1:5,
				ymode=log,
				ytick={0.01, 0.03, 0.10, 0.30, 1.00}, yticklabels={}, ymin=0.006, ymax=1.1];
				
				\pgfplotsforeachungrouped \n in {4,5,...,9}{
					\addplot table [x=k, y=a\n_r0_energy_norm, col sep=comma] {bh_sq_p3.csv};
				};
				
				\nextgroupplot[grid=major,
				title={\(p = 4\)},
				xlabel={\(k\)},
				restrict x to domain=1:5,
				ymode=log,
				ytick={0.01, 0.03, 0.10, 0.30, 1.00}, yticklabels={}, ymin=0.006, ymax=1.1,
				legend entries={4,5,6,7,8,9}, 
				legend style={at={(0.,-0.4)},anchor=north},
				legend columns=6,
				legend style={name=legend, draw=none}];
				
				\pgfplotsforeachungrouped \n in {4,5,...,9}{
					\addplot table [x=k, y=a\n_r0_energy_norm, col sep=comma] {bh_sq_p4.csv};
				};
				
				\nextgroupplot[grid=major,
				title={\(p = 5\)},
				xlabel={\(k\)},
				restrict x to domain=1:5,
				ymode=log,
				ytick={0.01, 0.03, 0.10, 0.30, 1.00}, yticklabels={}, ymin=0.006, ymax=1.1];
				
				\pgfplotsforeachungrouped \n in {4,5,...,9}{
					\addplot table [x=k, y=a\n_r0_energy_norm, col sep=comma] {bh_sq_p5.csv};
				};
				
			\end{groupplot}
			
			\node [left,align=left] (legendtitle) at (legend.west) {\(t_{max}\)};
			\node [fit=(legendtitle)(legend),draw,inner sep=0pt] {};
		\end{tikzpicture}
		\caption{Comparison of the \(V(k,k)\)-cycle convergence rates for nonnested FSAI smoothers, with polynomial degrees \(p\in \{2, 3, 4, 5\}\) for the local spaces in \cref{ex:biharmonic_square}. Note that, in the case \(p=2\), we use degree 3 for local spaces with support at the boundary.}
		\label{fig:A_convergence_rates_square_p}
	\end{figure}
	
\end{example}

\begin{example}\label{ex:anisotropic_biharmonic}

	For our second example, we introduce anisotropy in the biharmonic equation via a fourth order tensor \(\mathbb{D} \in \mathbb{R}^{2\times 2\times 2 \times 2}\) with the particular shape \(\mathbb{D} = S \otimes S\), for \(S\in \mathbb{R}^{2\times 2}\) the s.p.d. matrix imposing a scaling with factor \(\kappa\) along the direction with angle \(\theta\), i.e.
	\begin{equation*}
		S = \begin{pmatrix}
			\cos\theta & -\sin\theta \\ \sin\theta & \cos\theta
		\end{pmatrix} \begin{pmatrix}
			\kappa & 0 \\ 0 & 1
		\end{pmatrix} \begin{pmatrix}
		\cos\theta & \sin\theta \\ -\sin\theta & \cos\theta
		\end{pmatrix}.
	\end{equation*}
	The anisotropic biharmonic equation then reads
	\begin{equation*}
		\left\{
		\begin{aligned}
			S_{ij} S_{kl} u_{,ijkl} &= f \quad &&\text{ in } \Omega,\\
			u &= g_0 \quad &&\text{ on } \partial\Omega,\\
			\nabla u \cdot (S\mathbf{n}) &= g_1 \quad &&\text{ on } \partial\Omega,
		\end{aligned}
		\right.
	\end{equation*}
	and the bilinear and linear forms of its weak formulation following Nitsche's method consist in replacing \(\Delta w\) by \(\trace(S\, H(w))\), \(H(\cdot)\) being the hessian matrix of second order derivatives, and \(\nabla w \cdot \mathbf{n}\) by \(\nabla w \cdot (S\mathbf{n})\), for \(w \in \{u,v\}\), in the isotropic equation \cref{eq:biharmonic_nitsche}.

	In this anisotropic setting, we are particularly interested in visualizing the FSAI preconditioner adaptively generated by \cref{alg:adaptive_fsai}. For that matter, we focus on the sparsity pattern \(\mathcal{S}\) of the preconditioning matrix \(F^{\top} S^{-1} F\) in the nonnested case (at discretization level 7 and for \(p=3\)). For each row index \(k\), the column indices in the corresponding row \(\mathcal{S}_k\) can be understood as a ``FSAI neighborhood'' of associated patches in the homogeneously refined cover, \(\mathcal{N}_k = \{\omega_j \,\colon \, j\in \mathcal{S}_k\}\). If we center every \(\mathcal{N}_k\) around the corresponding \(\omega_k\) and then superimpose all of them, we can visualize the ``average neighborhood'' as we do in \cref{fig:anisotropic_adaptive_fsai} for \(t_{max} = 6\). As we can see, the FSAI neighborhoods do extend, on average, along the anisotropy axis associated to \(\kappa\)-scaling.

	\begin{figure}[tbhp]
		\centering
		\begin{tikzpicture}[scale=0.85]
			
			\fill [opacity=0.50000] (-0.25, -0.25) rectangle (0.25, 0.25);
			\fill [opacity=0.49609] (-0.25, 0.25) rectangle (0.25, 0.75);
			\fill [opacity=0.49222] (0.25, 0.25) rectangle (0.75, 0.75);
			\fill [opacity=0.49609] (0.25, -0.25) rectangle (0.75, 0.25);
			\fill [opacity=0.25879] (0.75, -0.25) rectangle (1.25, 0.25);
			\fill [opacity=0.02133] (1.25, -0.25) rectangle (1.75, 0.25);
			\fill [opacity=0.02396] (1.25, 0.25) rectangle (1.75, 0.75);
			\fill [opacity=0.23340] (0.75, 0.25) rectangle (1.25, 0.75);
			\fill [opacity=0.09659] (0.25, 1.25) rectangle (0.75, 1.75);
			\fill [opacity=0.34723] (0.25, 0.75) rectangle (0.75, 1.25);
			\fill [opacity=0.41174] (-0.25, 0.75) rectangle (0.25, 1.25);
			\fill [opacity=0.49609] (-0.25, -0.75) rectangle (0.25, -0.25);
			\fill [opacity=0.49222] (0.25, -0.75) rectangle (0.75, -0.25);
			\fill [opacity=0.14667] (0.75, -0.75) rectangle (1.25, -0.25);
			\fill [opacity=0.00931] (1.25, -0.75) rectangle (1.75, -0.25);
			\fill [opacity=0.13788] (0.75, 0.75) rectangle (1.25, 1.25);
			\fill [opacity=0.14590] (-0.25, 1.25) rectangle (0.25, 1.75);
			\fill [opacity=0.49222] (-0.75, -0.75) rectangle (-0.25, -0.25);
			\fill [opacity=0.49609] (-0.75, -0.25) rectangle (-0.25, 0.25);
			\fill [opacity=0.49222] (-0.75, 0.25) rectangle (-0.25, 0.75);
			\fill [opacity=0.30533] (-0.75, 0.75) rectangle (-0.25, 1.25);
			\fill [opacity=0.05847] (-0.75, 1.25) rectangle (-0.25, 1.75);
			\fill [opacity=0.25879] (-1.25, -0.25) rectangle (-0.75, 0.25);
			\fill [opacity=0.14667] (-1.25, 0.25) rectangle (-0.75, 0.75);
			\fill [opacity=0.02133] (-1.75, -0.25) rectangle (-1.25, 0.25);
			\fill [opacity=0.00931] (-1.75, 0.25) rectangle (-1.25, 0.75);
			\fill [opacity=0.02396] (-1.75, -0.75) rectangle (-1.25, -0.25);
			\fill [opacity=0.23340] (-1.25, -0.75) rectangle (-0.75, -0.25);
			\fill [opacity=0.03729] (-1.25, 0.75) rectangle (-0.75, 1.25);
			\fill [opacity=0.34723] (-0.75, -1.25) rectangle (-0.25, -0.75);
			\fill [opacity=0.41174] (-0.25, -1.25) rectangle (0.25, -0.75);
			\fill [opacity=0.30533] (0.25, -1.25) rectangle (0.75, -0.75);
			\fill [opacity=0.13788] (-1.25, -1.25) rectangle (-0.75, -0.75);
			\fill [opacity=0.09659] (-0.75, -1.75) rectangle (-0.25, -1.25);
			\fill [opacity=0.14590] (-0.25, -1.75) rectangle (0.25, -1.25);
			\fill [opacity=0.03729] (0.75, -1.25) rectangle (1.25, -0.75);
			\fill [opacity=0.00690] (-0.25, 1.75) rectangle (0.25, 2.25);
			\fill [opacity=0.05847] (0.25, -1.75) rectangle (0.75, -1.25);
			\fill [opacity=0.00690] (-0.25, -2.25) rectangle (0.25, -1.75);
			\fill [opacity=0.00415] (0.75, 1.25) rectangle (1.25, 1.75);
			\fill [opacity=0.00275] (1.75, -0.25) rectangle (2.25, 0.25);
			\fill [opacity=0.00275] (-2.25, -0.25) rectangle (-1.75, 0.25);
			\fill [opacity=0.00415] (-1.25, -1.75) rectangle (-0.75, -1.25);

			\draw [very thin,black,dashed] ({-2 * 0.5}, {-2 * 0.8660254}) -- ({2 * 0.5}, {2 * 0.8660254});
			
			\fill [color=black] (0,0) circle [radius=2pt];

			\fill [opacity=0.50000] (4.75, -0.25) rectangle (5.25, 0.25);
			\fill [opacity=0.49609] (4.75, 0.25) rectangle (5.25, 0.75);
			\fill [opacity=0.49222] (5.25, 0.25) rectangle (5.75, 0.75);
			\fill [opacity=0.49609] (5.25, -0.25) rectangle (5.75, 0.25);
			\fill [opacity=0.42084] (5.75, -0.25) rectangle (6.25, 0.25);
			\fill [opacity=0.15359] (6.25, -0.25) rectangle (6.75, 0.25);
			\fill [opacity=0.05060] (6.25, 0.25) rectangle (6.75, 0.75);
			\fill [opacity=0.49609] (4.75, -0.75) rectangle (5.25, -0.25);
			\fill [opacity=0.49222] (5.25, -0.75) rectangle (5.75, -0.25);
			\fill [opacity=0.36227] (5.75, -0.75) rectangle (6.25, -0.25);
			\fill [opacity=0.13696] (6.25, -0.75) rectangle (6.75, -0.25);
			\fill [opacity=0.27481] (5.75, 0.25) rectangle (6.25, 0.75);
			\fill [opacity=0.02585] (5.75, 0.75) rectangle (6.25, 1.25);
			\fill [opacity=0.14810] (5.25, 0.75) rectangle (5.75, 1.25);
			\fill [opacity=0.26459] (4.75, 0.75) rectangle (5.25, 1.25);
			\fill [opacity=0.49222] (4.25, -0.75) rectangle (4.75, -0.25);
			\fill [opacity=0.49609] (4.25, -0.25) rectangle (4.75, 0.25);
			\fill [opacity=0.49222] (4.25, 0.25) rectangle (4.75, 0.75);
			\fill [opacity=0.42084] (3.75, -0.25) rectangle (4.25, 0.25);
			\fill [opacity=0.36227] (3.75, 0.25) rectangle (4.25, 0.75);
			\fill [opacity=0.15359] (3.25, -0.25) rectangle (3.75, 0.25);
			\fill [opacity=0.13696] (3.25, 0.25) rectangle (3.75, 0.75);
			\fill [opacity=0.23825] (4.25, 0.75) rectangle (4.75, 1.25);
			\fill [opacity=0.02869] (4.75, 1.25) rectangle (5.25, 1.75);
			\fill [opacity=0.05060] (3.25, -0.75) rectangle (3.75, -0.25);
			\fill [opacity=0.27481] (3.75, -0.75) rectangle (4.25, -0.25);
			\fill [opacity=0.13495] (3.75, 0.75) rectangle (4.25, 1.25);
			\fill [opacity=0.14810] (4.25, -1.25) rectangle (4.75, -0.75);
			\fill [opacity=0.26459] (4.75, -1.25) rectangle (5.25, -0.75);
			\fill [opacity=0.23825] (5.25, -1.25) rectangle (5.75, -0.75);
			\fill [opacity=0.01587] (4.25, 1.25) rectangle (4.75, 1.75);
			\fill [opacity=0.02585] (3.75, -1.25) rectangle (4.25, -0.75);
			\fill [opacity=0.02869] (4.75, -1.75) rectangle (5.25, -1.25);
			\fill [opacity=0.13495] (5.75, -1.25) rectangle (6.25, -0.75);
			\fill [opacity=0.00293] (4.75, 1.75) rectangle (5.25, 2.25);
			\fill [opacity=0.00690] (5.25, 1.25) rectangle (5.75, 1.75);
			\fill [opacity=0.01587] (5.25, -1.75) rectangle (5.75, -1.25);
			\fill [opacity=0.00293] (4.75, -2.25) rectangle (5.25, -1.75);
			\fill [opacity=0.00690] (4.25, -1.75) rectangle (4.75, -1.25);
			\fill [opacity=0.00305] (6.75, -0.75) rectangle (7.25, -0.25);
			\fill [opacity=0.00305] (2.75, 0.25) rectangle (3.25, 0.75);
			\fill [opacity=0.00415] (6.25, 0.75) rectangle (6.75, 1.25);
			\fill [opacity=0.00415] (3.25, -1.25) rectangle (3.75, -0.75);

			\draw [very thin,black,dashed] ({5 - 2 * 0.8660254}, {2 * 0.5}) -- ({5 + 2 * 0.8660254}, {-2 * 0.5});
			
			\fill [color=black] (5,0) circle [radius=2pt];

			\fill [opacity=0.50000] (9.75, -0.25) rectangle (10.25, 0.25);
			\fill [opacity=0.49609] (9.75, 0.25) rectangle (10.25, 0.75);
			\fill [opacity=0.46259] (10.25, 0.25) rectangle (10.75, 0.75);
			\fill [opacity=0.49609] (10.25, -0.25) rectangle (10.75, 0.25);
			\fill [opacity=0.49219] (10.75, -0.25) rectangle (11.25, 0.25);
			\fill [opacity=0.26907] (11.25, -0.25) rectangle (11.75, 0.25);
			\fill [opacity=0.14484] (11.25, 0.25) rectangle (11.75, 0.75);
			\fill [opacity=0.49609] (9.75, -0.75) rectangle (10.25, -0.25);
			\fill [opacity=0.44730] (10.25, -0.75) rectangle (10.75, -0.25);
			\fill [opacity=0.36905] (10.75, -0.75) rectangle (11.25, -0.25);
			\fill [opacity=0.12286] (11.25, -0.75) rectangle (11.75, -0.25);
			\fill [opacity=0.40265] (10.75, 0.25) rectangle (11.25, 0.75);
			\fill [opacity=0.46259] (9.25, -0.75) rectangle (9.75, -0.25);
			\fill [opacity=0.49609] (9.25, -0.25) rectangle (9.75, 0.25);
			\fill [opacity=0.06140] (10.75, 0.75) rectangle (11.25, 1.25);
			\fill [opacity=0.10507] (10.25, 0.75) rectangle (10.75, 1.25);
			\fill [opacity=0.18030] (9.75, 0.75) rectangle (10.25, 1.25);
			\fill [opacity=0.44730] (9.25, 0.25) rectangle (9.75, 0.75);
			\fill [opacity=0.49219] (8.75, -0.25) rectangle (9.25, 0.25);
			\fill [opacity=0.36905] (8.75, 0.25) rectangle (9.25, 0.75);
			\fill [opacity=0.26907] (8.25, -0.25) rectangle (8.75, 0.25);
			\fill [opacity=0.12286] (8.25, 0.25) rectangle (8.75, 0.75);
			\fill [opacity=0.09409] (9.25, 0.75) rectangle (9.75, 1.25);
			\fill [opacity=0.03018] (9.75, 1.25) rectangle (10.25, 1.75);
			\fill [opacity=0.14484] (8.25, -0.75) rectangle (8.75, -0.25);
			\fill [opacity=0.40265] (8.75, -0.75) rectangle (9.25, -0.25);
			\fill [opacity=0.03839] (8.75, 0.75) rectangle (9.25, 1.25);
			\fill [opacity=0.10507] (9.25, -1.25) rectangle (9.75, -0.75);
			\fill [opacity=0.18030] (9.75, -1.25) rectangle (10.25, -0.75);
			\fill [opacity=0.09409] (10.25, -1.25) rectangle (10.75, -0.75);
			\fill [opacity=0.05045] (11.75, -0.25) rectangle (12.25, 0.25);
			\fill [opacity=0.06140] (8.75, -1.25) rectangle (9.25, -0.75);
			\fill [opacity=0.03018] (9.75, -1.75) rectangle (10.25, -1.25);
			\fill [opacity=0.03839] (10.75, -1.25) rectangle (11.25, -0.75);
			\fill [opacity=0.00983] (9.25, 1.25) rectangle (9.75, 1.75);
			\fill [opacity=0.00815] (10.25, 1.25) rectangle (10.75, 1.75);
			\fill [opacity=0.00293] (9.75, 1.75) rectangle (10.25, 2.25);
			\fill [opacity=0.00815] (9.25, -1.75) rectangle (9.75, -1.25);
			\fill [opacity=0.00983] (10.25, -1.75) rectangle (10.75, -1.25);
			\fill [opacity=0.00293] (9.75, -2.25) rectangle (10.25, -1.75);
			\fill [opacity=0.05045] (7.75, -0.25) rectangle (8.25, 0.25);
			\fill [opacity=0.00653] (11.75, 0.25) rectangle (12.25, 0.75);
			\fill [opacity=0.00653] (7.75, -0.75) rectangle (8.25, -0.25);
			
			\draw [very thin,black,dashed] ({10 - 2}, 0.) -- ({10 + 2}, 0.);
			
			\fill [color=black] (10,0) circle [radius=2pt];

		\end{tikzpicture}
		
		\caption{Average FSAI neighborhoods extracted from the rows of preconditioning matrix \(F^{\top} S^{-1} F\) with \(t_{max} = 6\), at discretization level \(7\), for the anisotropic problem from \cref{ex:anisotropic_biharmonic}. The left image corresponds to \(\theta = \frac{\pi}{3}\), the middle one to \(\theta = -\frac{\pi}{6}\), and the right one to \(\theta=0\) (with the respective axes illustrated by dashed lines going through the neighborhood's center). In all three cases we use an anisotropic scaling \(\kappa = 10\).}
		\label{fig:anisotropic_adaptive_fsai}
	\end{figure}
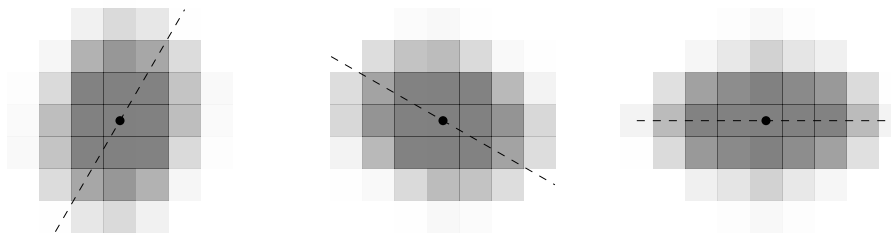
	
\end{example}

\begin{example}\label{ex:biharmonic_cube}

	For this new example, we consider again the biharmonic equation, but this time in the unit cube domain \(\Omega = (0, 1)^3\), where the weak formulation remains unmodified. Given the larger size of local spaces with respect to the two-dimensional case, we restrict \(t_{max} \leq 4\) for FSAI adaptivity and use 4 as the finest discretization level (keeping 2 as the coarsest level). In this case, each space consists of \(2^{3l}\) local spaces. As reference manufactured solution we pick in this case
	\begin{equation*}
		\hat{u}(x, y, z) = \cos(2 \pi x) \sin(2 \pi y) \sin(2 \pi z).
	\end{equation*}
	
	This time we iterate until the \(L^2(\Omega)\) norm of the error is below \(10^{-7}\) or a maximum number of 100 iterations is reached. Given the cost of using large polynomial degrees in a 3-dimensional setting (local spaces are of size \(\binom{p+3}{3}\), i.e.\ 20 for \(p=3\) and 35 for \(p=4\), and this size is multiplied by up to \(t_{max}\) for the blocks to be inverted at FSAI's setup stage), we provide only convergence rates for \(p = 2\) (in \cref{fig:A_convergence_rates_cube}, where as before we refine the polynomial degree to 3 for patches intersecting the boundary) and \(p = 3\) (in \cref{fig:A_convergence_rates_cube_p3}). The combination of discretization level 4 with polynomial degree 3 yields \(81,920\) degrees of freedom. As we can observe, the results for \(p=3\) are below those of \(p=2\) both for nested and nonnested FSAI. We remark the similar convergence rate histories for nested FSAI with \(t_{max} = c\) and nonnested FSAI with \(t_{max} = 2c\), pointing at the similarity of the FSAI \(F\) matrices when they are constructed in 2 levels with \(c\) non-zero blocks per row, or in a single level with \(2c\) non-zero blocks per row.
	
	Finally, in this case we choose \(p = 3\) and nested FSAI for the comparison with Richardson smoothing and with nonsymmetric cycles, which we provide simultaneously in \cref{fig:A_convergence_rates_cube_p3_comp}. As before, we clearly observe the superiority of Chebyshev smoothing with respect to Richardson smoothing, and of non-symmetric cycles with respect to symmetric ones.
	
	\begin{figure}[tbhp]
		\centering
		\begin{tikzpicture}
			\begin{groupplot}[group style={group size= 2 by 1, horizontal sep=0.5cm},height=5cm,width=6.75cm]
				\nextgroupplot[grid=major,
				title={Adaptive FSAI},
				xlabel={\(k\)}, ylabel={\(\rho_{A}\)},
				ymode=log,
				ytick={0.12, 0.25, 0.50, 1.00}, yticklabels={0.12, 0.25, 0.50, 1.00}, ymin=0.10, ymax=1.05,
				legend entries={\(1\),\(2\),\(3\),\(4\)}, legend pos=south west,
				legend style={name=legend, draw=none}, legend columns=2];
				
				\pgfplotsforeachungrouped \n in {1,2,3,4}{
					\addplot table [x=k, y=a\n_r0_energy_norm, col sep=comma] {bh_cube_pbdry.csv};
				};
				
				\nextgroupplot[grid=major,
				title={Nested adaptive FSAI},
				xlabel={\(k\)},
				ymode=log,
				ytick={0.12, 0.25, 0.50, 1.00}, yticklabels={}, ymin=0.10, ymax=1.05];
				
				\pgfplotsforeachungrouped \n in {1,2,3,4}{
					\addplot table [x=k, y=a\n_r1_energy_norm, col sep=comma] {bh_cube_pbdry.csv};
				};
				
			\end{groupplot}
			
			\node [above,text width=1.82cm,fill=white,align=center] (legendtitle) at (legend.north) {\( t_{max}\)};
			\node [fit=(legendtitle)(legend),draw,inner sep=0pt] {};
		\end{tikzpicture}
		\caption{\(V(k,k)\)-cycle convergence rates for \cref{ex:biharmonic_cube}, with local polynomial spaces of degree \(p = 2\) (refined to \(p = 3\) for patches intersecting the boundary).}
		\label{fig:A_convergence_rates_cube}
	\end{figure}
	
	\begin{figure}[tbhp]
		\centering
		\begin{tikzpicture}
			\begin{groupplot}[group style={group size= 2 by 1, horizontal sep=0.5cm},height=5cm,width=6.75cm]
				\nextgroupplot[grid=major,
				title={Adaptive FSAI},
				xlabel={\(k\)}, ylabel={\(\rho_{A}\)},
				ymode=log,
				ytick={0.06, 0.12, 0.25, 0.50, 1.00}, yticklabels={0.06, 0.12, 0.25, 0.50, 1.00}, ymin=0.05, ymax=1.05,
				legend entries={\(1\),\(2\),\(3\),\(4\)}, legend pos=south west,
				legend style={name=legend, draw=none}, legend columns=2];
				
				\pgfplotsforeachungrouped \n in {1,2,3,4}{
					\addplot table [x=k, y=a\n_r0_energy_norm, col sep=comma] {bh_cube_p3.csv};
				};
				
				\nextgroupplot[grid=major,
				title={Nested adaptive FSAI},
				xlabel={\(k\)},
				ymode=log,
				ytick={0.06, 0.12, 0.25, 0.50, 1.00}, yticklabels={}, ymin=0.05, ymax=1.05];
				
				\pgfplotsforeachungrouped \n in {1,2,3,4}{
					\addplot table [x=k, y=a\n_r1_energy_norm, col sep=comma] {bh_cube_p3.csv};
				};
				
			\end{groupplot}
			
			\node [above,text width=1.82cm,fill=white,align=center] (legendtitle) at (legend.north) {\( t_{max}\)};
			\node [fit=(legendtitle)(legend),draw,inner sep=0pt] {};
		\end{tikzpicture}
		\caption{\(V(k,k)\)-cycle convergence rates for \cref{ex:biharmonic_cube}, with local polynomial spaces of degree \(p = 3\) (compare with \cref{fig:A_convergence_rates_cube} to assess the impact of \(p\)-refinement).}
		\label{fig:A_convergence_rates_cube_p3}
	\end{figure}
	
	\begin{figure}[tbhp]
		\centering
		\begin{tikzpicture}
			\begin{groupplot}[group style={group size= 3 by 1, horizontal sep=0.25cm},height=4.5cm,width=5.25cm]
				\nextgroupplot[grid=major,
				title={\(V(k,k)\)-Richardson},
				xlabel={\(k\)}, ylabel={\(\rho_{A}\)},
				ymode=log,
				ytick={0.03, 0.06, 0.12, 0.25, 0.50, 1.00}, yticklabels={0.03, 0.06, 0.12, 0.25, 0.50, 1.00}, ymin=0.025, ymax=1.05,
				legend entries={\(1\),\(2\),\(3\),\(4\)}, legend pos=south west,
				legend style={name=legend, draw=none}, legend columns=2];
				
				\pgfplotsforeachungrouped \n in {1,2,3,4}{
					\addplot table [x=k, y=a\n_r1_energy_norm, col sep=comma] {bh_cube_richardson_p3.csv};
				};
				
				\nextgroupplot[grid=major,
				title={\(V(k,k)\)-Chebyshev},
				xlabel={\(k\)},
				ymode=log,
				ytick={0.03, 0.06, 0.12, 0.25, 0.50, 1.00}, yticklabels={}, ymin=0.025, ymax=1.05];
				
				\pgfplotsforeachungrouped \n in {1,2,3,4}{
					\addplot table [x=k, y=a\n_r1_energy_norm, col sep=comma] {bh_cube_p3.csv};
				};
				
				\nextgroupplot[grid=major,
				title={\(V(2k,0)\)-Chebyshev},
				xlabel={\(k\)},
				ymode=log,
				ytick={0.03, 0.06, 0.12, 0.25, 0.50, 1.00}, yticklabels={}, ymin=0.025, ymax=1.05];
				
				\pgfplotsforeachungrouped \n in {1,2,3,4}{
					\addplot table [x=k, y=a\n_r1_energy_norm, col sep=comma] {bh_cube_nonsym_p3.csv};
				};
				
			\end{groupplot}
			
			\node [above,text width=1.82cm,fill=white,align=center] (legendtitle) at (legend.north) {\( t_{max}\)};
			\node [fit=(legendtitle)(legend),draw,inner sep=0pt] {};
		\end{tikzpicture}
		\caption{Comparison of Richardson and Chebyshev \(V(k,k)\)-cycle, and Chebyshev \(V(2k,0)\)-cycle convergence rates, for \cref{ex:biharmonic_cube}, with nested FSAI smoothers and local polynomial spaces of degree \(p = 3\).}
		\label{fig:A_convergence_rates_cube_p3_comp}
	\end{figure}
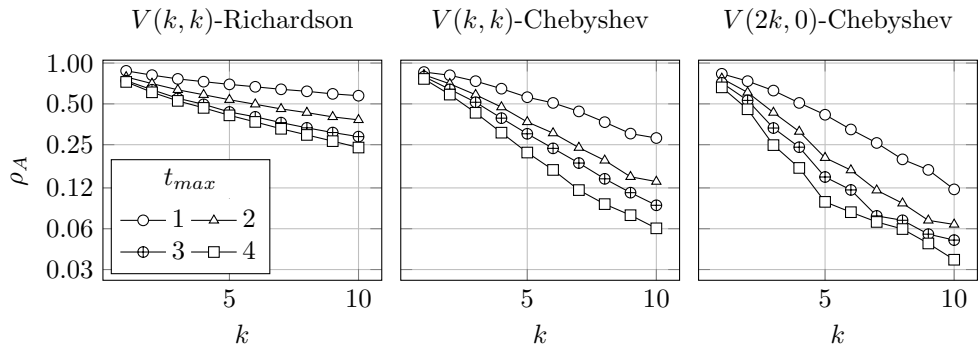

\end{example}

\begin{example}\label{ex:triharmonic_square}
	For our last example we further increase the order of the problem and aim to solve the \emph{triharmonic} equation
	\begin{equation*}
		\left\{
		\begin{aligned}
			-\Delta^3 u &= f \quad &&\text{ in } \Omega,\\
			u &= g_0 \quad &&\text{ on } \partial\Omega,\\
			\nabla u \cdot \mathbf{n} &= g_1 \quad &&\text{ on } \partial\Omega,\\
			\Delta u &= g_2 \quad &&\text{ on } \partial\Omega.
		\end{aligned}
		\right.
	\end{equation*}
	For a discrete space \(V_n\subset H^3(\Omega)\), Nitsche's method in this case yields a weak formulation with bilinear and linear forms
	\begin{align*}
		a_n(v, u) =& \int_{\Omega} \nabla (\Delta u) \cdot \nabla (\Delta v) \diff\Omega + \int_{\partial\Omega} \left(\gamma_n^{(0)} u v - v \nabla(\Delta^2 u) \cdot \mathbf{n} - u \nabla(\Delta^2 v) \cdot \mathbf{n} \right) \diff\Gamma +\\
		&+ \int_{\partial\Omega} \left( \gamma_n^{(1)} (\nabla u \cdot \mathbf{n}) (\nabla v \cdot \mathbf{n}) + \Delta^2 u (\nabla v \cdot \mathbf{n}) + \Delta^2 v (\nabla u \cdot \mathbf{n}) \right) \diff\Gamma +\\
		&+ \int_{\partial\Omega} \left( \gamma_n^{(2)} \Delta u \Delta v - \Delta v \nabla (\Delta u) \cdot \mathbf{n} - \Delta u \nabla (\Delta v) \cdot \mathbf{n} \right) \diff\Gamma ,\\
		\ell_n(v) =& \int_{\Omega} f v \diff\Omega + \int_{\partial\Omega} g_0 \left(\gamma_n^{(0)} v - \nabla(\Delta^2 v) \cdot \mathbf{n} \right) \diff\Gamma +\\
		&+ \int_{\partial\Omega} g_1 \left( \gamma_n^{(1)} (\nabla v \cdot \mathbf{n}) + \Delta^2 v \right) \diff\Gamma + \int_{\partial\Omega} g_2 \left( \gamma_n^{(2)} \Delta v - \nabla (\Delta v) \cdot \mathbf{n} \right) \diff\Gamma.
	\end{align*}
	
	In this case, we choose 5 as the finest discretization level, and \(p\in\{4,5\}\) for the local spaces (for \(p=4\), we refine the local spaces with support at the boundary to \(p=5\)). With respect to FSAI adaptivity, we restrict \(t_{max} \in [6, 9]\). In \cref{fig:A_convergence_rates_th}, we provide the measured convergence rates for the case \(p = 4\), and in \cref{fig:A_convergence_rates_th_5}, those for \(p = 5\). We finally choose the \(p = 4\) case with nested FSAI smoothers for the comparison with nonsymmetric cycles, whose results we provide in \cref{fig:sym_vs_nonsym_chebyshev_th}, showing once again the superiority of the \(V(2k, 0)\)-cycles. In all cases, we iterate until the \(L^2(\Omega)\) norm of the error is below \(10^{-7}\) or a maximum number of 100 iterations is reached. We note that, for this problem, Richardson smoothing did not yield a convergent cycle for any of our considered FSAI settings in the \(p = 4\) case.
	
	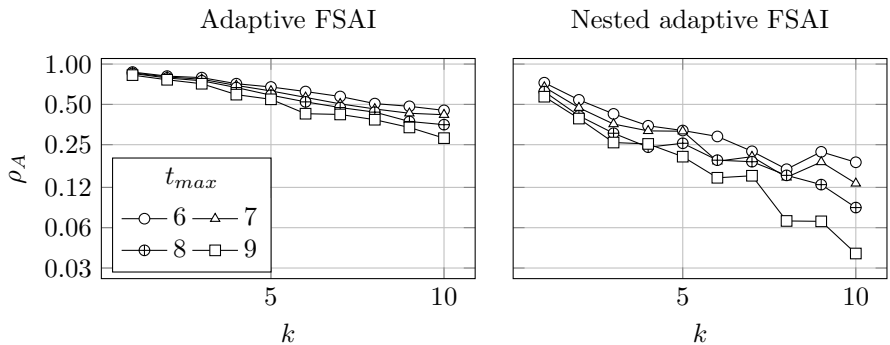
\begin{figure}[tbhp]
		\centering
		\begin{tikzpicture}
			\begin{groupplot}[group style={group size= 2 by 1, horizontal sep=0.5cm},height=4.5cm,width=6.5cm]
				\nextgroupplot[grid=major,
				title={Adaptive FSAI},
				xlabel={\(k\)}, ylabel={\(\rho_{A}\)},
				ymode=log,
				ytick={0.03, 0.06, 0.12, 0.25, 0.50, 1.00}, yticklabels={0.03, 0.06, 0.12, 0.25, 0.50, 1.00}, ymin=0.025, ymax=1.1,
				legend entries={6,7,8,9}, legend pos=south west,
				legend style={name=legend, draw=none}, legend columns=2];
				
				\pgfplotsforeachungrouped \n in {6,7,...,9}{
					\addplot table [x=k, y=a\n_r0_energy_norm, col sep=comma] {th_sq_p4.csv};
				};
				
				\nextgroupplot[grid=major,
				title={Nested adaptive FSAI},
				xlabel={\(k\)},
				ymode=log,
				ytick={0.03, 0.06, 0.12, 0.25, 0.50, 1.00}, yticklabels={}, ymin=0.025, ymax=1.1];
				
				\pgfplotsforeachungrouped \n in {6,7,...,9}{
					\addplot table [x=k, y=a\n_r1_energy_norm, col sep=comma] {th_sq_p4.csv};
				};
				
			\end{groupplot}
			
			\node [above,text width=1.82cm,fill=white,align=center] (legendtitle) at (legend.north) {\( t_{max}\)};
			\node [fit=(legendtitle)(legend),draw,inner sep=0pt] {};
		\end{tikzpicture}
		\caption{\(V(k,k)\)-cycle convergence rates for \cref{ex:triharmonic_square}, with local polynomial spaces of degree \(p = 4\) (refined to \(p = 5\) for patches intersecting the boundary).}
		\label{fig:A_convergence_rates_th}
	\end{figure}
	
	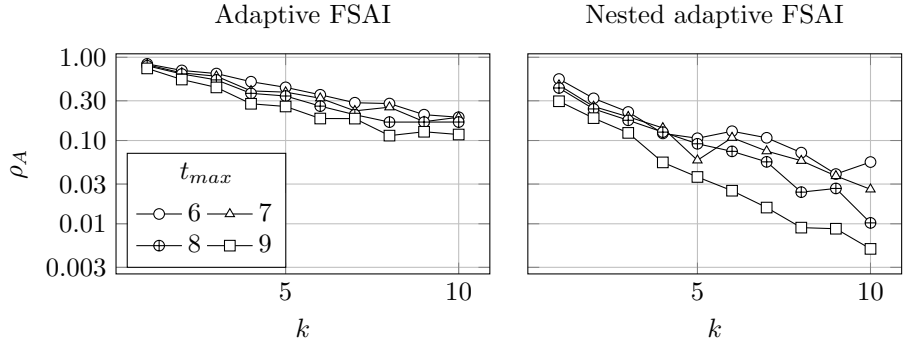
\begin{figure}[tbhp]
		\centering
		\begin{tikzpicture}
			\begin{groupplot}[group style={group size= 2 by 1, horizontal sep=0.5cm},height=4.5cm,width=6.5cm]
				\nextgroupplot[grid=major,
				title={Adaptive FSAI},
				xlabel={\(k\)}, ylabel={\(\rho_{A}\)},
				ymode=log,
				ytick={0.003, 0.01, 0.03, 0.10, 0.30, 1.00}, yticklabels={0.003, 0.01, 0.03, 0.10, 0.30, 1.00}, ymin=0.0025, ymax=1.1,
				legend entries={6,7,8,9}, legend pos=south west,
				legend style={name=legend, draw=none}, legend columns=2];
				
				\pgfplotsforeachungrouped \n in {6,7,...,9}{
					\addplot table [x=k, y=a\n_r0_energy_norm, col sep=comma] {th_sq_p5.csv};
				};
				
				\nextgroupplot[grid=major,
				title={Nested adaptive FSAI},
				xlabel={\(k\)},
				ymode=log,
				ytick={0.003, 0.01, 0.03, 0.10, 0.30, 1.00}, yticklabels={}, ymin=0.0025, ymax=1.1];
				
				\pgfplotsforeachungrouped \n in {6,7,...,9}{
					\addplot table [x=k, y=a\n_r1_energy_norm, col sep=comma] {th_sq_p5.csv};
				};
				
			\end{groupplot}
			
			\node [above,text width=1.82cm,fill=white,align=center] (legendtitle) at (legend.north) {\( t_{max}\)};
			\node [fit=(legendtitle)(legend),draw,inner sep=0pt] {};
		\end{tikzpicture}
		\caption{\(V(k,k)\)-cycle convergence rates for \cref{ex:triharmonic_square}, with local polynomial spaces of degree \(p = 5\) (compare with \cref{fig:A_convergence_rates_th} to assess the impact of \(p\)-refinement).}
		\label{fig:A_convergence_rates_th_5}
	\end{figure}
	
	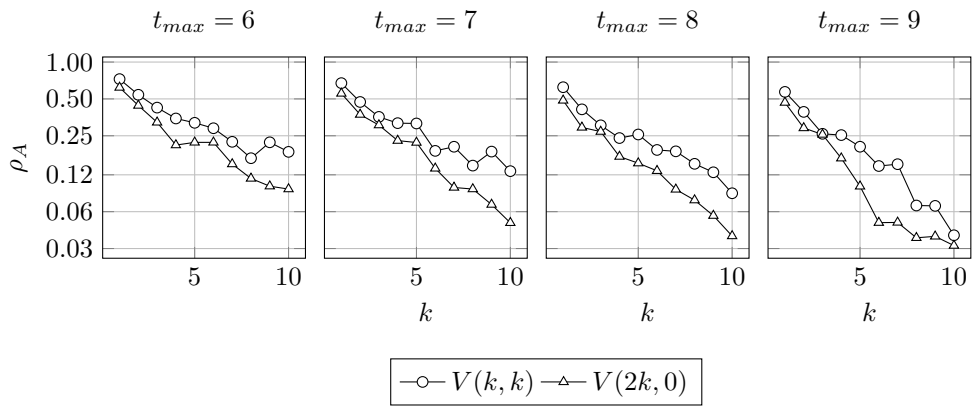
\begin{figure}[!htbp]
		\centering
		\begin{tikzpicture}
			\begin{groupplot}[group style={group size= 4 by 1, horizontal sep=0.25cm},height=4.25cm,width=4.25cm]
				
				\nextgroupplot[grid=major,
				title={\(t_{max}=6\)},
				ylabel={\(\rho_{A}\)},
				ymode=log,
				ytick={0.03, 0.06, 0.12, 0.25, 0.50, 1.00}, yticklabels={0.03, 0.06, 0.12, 0.25, 0.50, 1.00}, ymin=0.025, ymax=1.1];
				
				\addplot table [x=k, y=a6_r1_energy_norm, col sep=comma] {th_sq_p4.csv};
				
				\addplot table [x=k, y=a6_r1_energy_norm, col sep=comma] {th_sq_nonsym_p4.csv};
				
				\nextgroupplot[grid=major,
				title={\(t_{max}=7\)},
				xlabel={\(k\)},
				ymode=log,
				ytick={0.03, 0.06, 0.12, 0.25, 0.50, 1.00}, yticklabels={}, ymin=0.025, ymax=1.1];
				
				\addplot table [x=k, y=a7_r1_energy_norm, col sep=comma] {th_sq_p4.csv};
				
				\addplot table [x=k, y=a7_r1_energy_norm, col sep=comma] {th_sq_nonsym_p4.csv};
				
				\nextgroupplot[grid=major,
				title={\(t_{max}=8\)},
				xlabel={\(k\)},
				ymode=log,
				ytick={0.03, 0.06, 0.12, 0.25, 0.50, 1.00}, yticklabels={}, ymin=0.025, ymax=1.1,
				legend entries={\(V(k,k)\)\\\(V(2k,0)\)\\}, 
				legend style={at={(0.,-0.5)},anchor=north},
				legend columns=2,
				legend style={name=legend, draw=none}];
				
				\addplot table [x=k, y=a8_r1_energy_norm, col sep=comma] {th_sq_p4.csv};
				
				\addplot table [x=k, y=a8_r1_energy_norm, col sep=comma] {th_sq_nonsym_p4.csv};
				
				\nextgroupplot[grid=major,
				title={\(t_{max}=9\)},
				xlabel={\(k\)},
				ymode=log,
				ytick={0.03, 0.06, 0.12, 0.25, 0.50, 1.00}, yticklabels={}, ymin=0.025, ymax=1.1];
				
				\addplot table [x=k, y=a9_r1_energy_norm, col sep=comma] {th_sq_p4.csv};
				
				\addplot table [x=k, y=a9_r1_energy_norm, col sep=comma] {th_sq_nonsym_p4.csv};
				
			\end{groupplot}
			
			\node [fit=(legend),draw,inner sep=0pt] {};
		\end{tikzpicture}
		\caption{Comparison of the \(V(k,k)\) and \(V(2k,0)\)-cycle convergence rates for \cref{ex:triharmonic_square}, with nested adaptive FSAI smoothers and local polynomial spaces of degree \(p = 4\) (refined to \(p = 5\) for patches intersecting the boundary).}
		\label{fig:sym_vs_nonsym_chebyshev_th}
	\end{figure}

\end{example}

\section{Conclusions}\label{sec:conclusions}

We have explored the flexibility of FSAI smoothers with respect to adaptivity and nestedness, relying on PUM discretizations of the biharmonic and triharmonic equations. We have shown how their smoothing capability improves with increasing density of the preconditioning matrix, which we achieved either by allowing more non-zero entries per row in the adaptive algorithm, or by nesting an additional FSAI preconditioner into the first one (or by a combination of both). We have shown that their effectiveness increases with the polynomial degree of the discrete space, and that the adaptive pattern construction allows to capture anisotropies in the PDE.

We have also shown that smoothing based on the Chebyshev iteration of the fourth kind yields better convergence rates than the usual Richardson smoothing, and, for a limited number of smoothing steps, even allows convergence in cases where Richardson smoothing does not. Additionally, we have confirmed that, as pointed out already in the literature, non-symmetric \(V(2k,0)\) cycles yield faster convergence rates than their symmetric \(V(k,k)\) counterparts. Nevertheless, we note that when using the multilevel iteration as a preconditioner within some iterative solver, the use of \(V(2k,0)\) cycles prevents the use of a standard conjugate gradient solver, which has to be replaced by some solver accepting a non-symmetric preconditioner (e.g.\ BiCGStab).

With respect to our algorithms, we have introduced a simple but effective adaptive construction for the FSAI preconditioner, for matrices with a certain block structure. This includes the matrices arising from a PUM discretization, but could be extended to different scenarios. Additionally, we have provided a new formulation of the Chebyshev iteration of the fourth kind, which in our opinion is even simpler than the one originally given by Lottes \cite{Lottes2023}. We agree with him that this new Chebyshev iteration deserves more wide-spread recognition in contrast to the usual Richardson smoothing, and we hope that our work will make a contribution to that, even if a small one.

\bibliographystyle{ownabbrvnat}
\bibliography{references}

\end{document}